\documentclass[11pt]{article}
\usepackage{amssymb}
\usepackage{amsfonts}
\usepackage{amsmath,latexsym,amssymb,amsfonts,amsbsy, amsthm}
\usepackage[usenames]{color}
\usepackage{xspace,colortbl}
\usepackage{mathrsfs}
\usepackage[colorlinks,linkcolor=blue,citecolor=blue]{hyperref}
\allowdisplaybreaks

\newcommand{\beq}{\begin{equation}}

\newcommand{\eeq}{\end{equation}}
\newcommand{\ben}{\begin{eqnarray}}
\newcommand{\een}{\end{eqnarray}}
\newcommand{\beno}{\begin{eqnarray*}}
\newcommand{\eeno}{\end{eqnarray*}}

\newtheorem{thm}{Theorem}[section]

\newtheorem{lem}[thm]{Lemma}
\newtheorem{prop}[thm]{Proposition}
\newtheorem{coro}[thm]{Corollary}
\newtheorem{rmk}[thm]{Remark}

\renewcommand{\theequation}{\thesection.\arabic{equation}}

\title{\textbf{On the De Giorgi type conjecture for an elliptic system
modeling phase separation}}
\author{Kelei Wang
\thanks{This work was done while the author was a postdoc at the University of Sydney and was
supported by the Australian Research Council. I would like to express my sincere thanks to
Prof. E. N. Dancer, S. Terracini and J. C. Wei for valuable
discussions related to this problem. I am also grateful to referees for their careful reading and useful suggestions.}
\\
{\small  Wuhan Institute of Physics and Mathematics,}\\
{\small Chinese Academy of Sciences, Wuhan 430071, China}\\
{\small wangkelei@wipm.ac.cn } \ }
\date{}

\begin{document}
\maketitle
\begin{abstract}
In this paper we study the one dimensional symmetry problem of
entire solutions to the problem
\[\Delta u=uv^2,\ \ \Delta v=vu^2,\ \ u,v>0~~\text{in}~~\mathbb{R}^n,\]
for all $n\geq 2$. We prove that, if a solution $(u,v)$ is a local
minimizer and has a linear growth at infinity, then it is one
dimensional, i.e. depending only on one variable. In the proof we
also obtain the global Lipschitz continuity of solutions only under
the linear growth assumption.
\end{abstract}

\noindent {\sl Keywords:} {\small elliptic systems, phase
separation, one dimensional symmetry, sliding method.}\

\vskip 0.2cm

\noindent {\sl AMS Subject Classification (2000):} {\small 35B06,
35B08, 35B25, 35J91.}

\renewcommand{\theequation}{\thesection.\arabic{equation}}
\setcounter{equation}{0}

\tableofcontents

\section{Introduction}
\numberwithin{equation}{section}
 \setcounter{equation}{0}
 In this paper, we study the one dimensional symmetry problem for
solutions of the following two component elliptic system in
$\mathbb{R}^n$:
\begin{equation}\label{equation}
\left\{ \begin{aligned}
 &\Delta u=uv^2, \\
 &\Delta v=vu^2.
                          \end{aligned} \right.
\end{equation}
All of the solutions considered in this paper are positive classical
solutions, that is, $u>0$ and $v>0$ and they are smooth.
\par
We say a function $u$, defined on $\mathbb{R}^n$, is one dimensional
if there exists a unit vector $e\in\mathbb{R}^n$ and a function $f$ defined on
$\mathbb{R}^1$, such that $u(x)\equiv f(x\cdot e)$.

 The system
\eqref{equation} arises from many fields in physics such as
Bose-Einstein condensation and nonlinear optics.
It is used to describe the ``Phase Separation" phenomena.  For more background, see \cite{blwz, DWZ2011, NTTV,
TT2011} and references therein.

In Berestycki-Lin-Wei-Zhao \cite{blwz}, inspired by the De Giorgi
conjecture for the Allen-Cahn equation (cf. \cite{Savin}), they ask
whether there is one dimensional symmetry for entire solutions of
\eqref{equation}. In \cite{blwz} they also proved the existence, symmetry
and nondegeneracy of the solution to the one-dimensional problem
\begin{equation}
u^{\prime\prime}= uv^2, \ \ v^{\prime\prime}=v u^2, \ \ u, v>0
\ ~~\mbox{in}~~\mathbb{R}.
\end{equation}
In particular they showed that entire solutions of this problem are
reflectionally symmetric, i.e., there exists $x_0$ such that $
u(x-x_0)= v(x_0-x)$. In \cite{BTWW}, together with Berestycki, Terracini and Wei, the author also proved that, up to a
scaling and translation, this entire solution is unique. This
solution can be trivially extended to $\mathbb{R}^n$ for all $n\geq
2$, which gives a solution of \eqref{equation} with a linear growth.
We also note that, it was proved in  Noris-Tavares-Terracini-Verzini \cite{NTTV} that the linear
growth is the lowest possible for solutions to \eqref{equation}.
More precisely, if there exists $\alpha \in (0,1)$ such that
\begin{equation*}
u(x)+v(x)\leq C (1+|x|)^{\alpha},\ \ \ \mbox{in}\ \mathbb{R}^n,
\end{equation*}
then $u, v \equiv 0$.

Unlike the Allen-Cahn equation, 
where minimal hypersurfaces play an important role in the limiting
problem, the limiting problem of \eqref{equation} is related to
harmonic functions. One typical result is (cf. Dancer-Wang-Zhang \cite{DWZ2011} and
Tavares-Terracini \cite{TT2011}), as $\kappa\to+\infty$, any sequence of uniformly
bounded solutions $(u_\kappa, v_\kappa)$ to the problem
\begin{equation}\label{equation scaled}
\left\{
\begin{aligned}
 &\Delta u_\kappa=\kappa u_\kappa v_\kappa^2, \\
 &\Delta v_\kappa=\kappa v_\kappa u_\kappa^2,
                          \end{aligned} \right.
\end{equation}
converges uniformly (up to a subsequence of $\kappa\to+\infty$) to
$(w^+,w^-)$. Here $w$ is a harmonic function and $w^+$ is its
positive part, $w^-=(-w)^+$ the negative part.

Note that solutions of \eqref{equation scaled} are critical points
of the energy functional (under suitable boundary conditions)
\begin{equation}\label{functional}
E_\kappa(u,v):=\int|\nabla u_\kappa|^2+|\nabla v_\kappa|^2+\kappa
u_\kappa^2v_\kappa^2.
\end{equation}

 For a solution of
\eqref{equation} with a linear growth at infinity, by performing
suitable blowing-down procedure (see Section 3 for details), the
blowing down sequence converges to $((e\cdot x)^+,(e\cdot x)^-)$ for
some constant vector $e$. This means that the level sets of $u-v$
are asymptotically flat at infinity. Thus it is very natural to
conjecture that these level sets are flat and the De Giorgi type
conjecture holds under this linear growth condition.

In \cite{blwz} and \cite{BTWW} (see also Farina \cite{F2} for related results), several results in this direction were
obtained when the space dimension $n=2$. These works assume the
solution satisfies a monotone condition or a stability condition. In
\cite{BTWW}, we also proved that, when $n=2$, for every $d\geq 2$,
there is a solution of \eqref{equation}, such that $u-v$ is
asymptotic to $\text{Re}(x+\sqrt{-1}y)^d$ (i.e. a homogeneous
harmonic polynomial of degree $d$) at infinity. These examples show
that we cannot remove the linear growth assumption in the De Giorgi
type conjecture.

 In this paper we prove the following result
 for all $n\geq 2$.
\begin{thm}\label{main result}
If $(u,v)$ is a local minimizer of the problem \eqref{equation} in
$\mathbb{R}^n$, and there exists a constant $C>0$ such that for any
$x\in\mathbb{R}^n$,
\begin{equation}\label{linear growth}
u(x)+v(x)\leq C(1+|x|),
\end{equation}
then $(u,v)$ is one dimensional.
\end{thm}
Here a local minimizer means, for every smooth functions
$(\bar{u},\bar{v})$ such that $(\bar{u},\bar{v})=(u,v)$ outside a
ball $B_R(0)$, we have
\begin{equation}\label{local energy minimizer}
\int_{B_R(0)}|\nabla u|^2+|\nabla v|^2+u^2v^2 \leq
\int_{B_R(0)}|\nabla \bar{u}|^2+|\nabla
\bar{v}|^2+\bar{u}^2\bar{v}^2.
\end{equation}
This is only a technical assumption and we believe it can be
removed. In our proof, it is only used to compare the difference of
energy between $(u,v)$ and a harmonic replacement (see Proposition
\ref{lem comparison of the energy}). Note that minimizers are always stable, so when $n=2$ the above result is contained in the result in \cite{BTWW} for stable solutions. Let us mention that in a recent preprint \cite{F-S}, A. Farina and N. Soave solved the Gibbons conjecture for this class of problem. There they used the monotonicity condition rather than the minimizing condition. The relations between these conditions still need further exploration.

Next we explain briefly the strategy of our proof and the
organization of this paper. In Section 2 we collect some useful
results such as the Almgren type monotonicity formula for solutions
to \eqref{equation}. In Section 3 we perform a blowing down analysis
and show that $(u,v)$ is asymptotically flat at infinity.

Section 4 is devoted to the proof of an Alt-Caffarelli-Friedman
monotonicity formula for solutions of \eqref{equation} with linear growth at infinity (see Theorem
\ref{thm ACF monotonicity}), which says
\[e^{-Cr^{-1/2}}r^{-4}\left(\int_{B_r(0)}\frac{|\nabla
u(y)|^2+u(y)^2v(y)^2}{|y|^{n-2}}dy\right)\left( \int_{B_r(0)}\frac{|\nabla
v(y)|^2+u(y)^2v(y)^2}{|y|^{n-2}}dy\right)\] is non-decreasing in $r>1$.
This can be seen as a sharp form of the Alt-Caffarelli-Friedman monotonicity formula
developed by Noris-Tavares-Terracini-Verzini in \cite{NTTV}. To
achieve this we use a special Steiner symmetrization rearrangement for the two component
functions $(u,v)$ on the unit sphere $\mathbb{S}^{n-1}$. This allow
us to reduce a minimization problem in higher dimensional sphere to
a one dimension problem, where by the results in \cite{blwz} we have
better controls such as uniform Lipschitz continuity. (Note that at
present the uniform Lipschitz continuity of solutions of
\eqref{equation scaled} is only known when the space dimension
$n=1$.) Similar ideas have already been used in \cite{BTWW} (see
Theorem 5.6 therein).

 This monotonicity formula can be used to
give a lower bound of the growth rate of $(u,v)$. For example, for a
solution $(u,v)$ with a linear growth, using this monotonicity
formula we can prove a nondegeneracy result (Corollary \ref{coro
5.4}): there exists a constant $C$ such that
\[\int_{B_R(0)}u+v\geq CR^{n+1}.\]
More importantly, this monotonicity formula and some of its
consequences hold for all $x\in\{u=v\}$, with a constant
$C$ independent of $x$ and the radius $R$. This fact, together with the results in Section 3, implies that
at large scale (uniformly with respect to the base point
$x\in\{u=v\}$), $(u,v)$ is close to $((e\cdot x)^+,(e\cdot x)^-)$
for some unit vector $e$. This then enables us to prove the global
Lipschitz continuity of $u$ and $v$ (see Theorem \ref{thm 6.1}
).

In Section 6, we use the global Lipschitz property and the locally
energy minimizing property of $(u,v)$ to deduce the following
crucial estimate
\[\int_{B_R(x)}|\nabla (u-v-\varphi)|^2\leq CR^{n-1/2},\]
where $C$ is independent of $x$ and $R$, and $\varphi$ is the
harmonic extension of $u-v$ from $\partial B_R(x)$ to $B_R(x)$. This
is the only place where we need the energy minimizing property of
$(u,v)$.

The exponent $1/2$ in this estimate implies the existence of a
unique vector $e$ such that for every $x_0\in\mathbb{R}^n$
\[\lim\limits_{R\to+\infty}R^{-1}u(x_0+Rx)=(e\cdot x)^+,\ \ \ \
\lim\limits_{R\to+\infty}R^{-1}v(x_0+Rx)=(e\cdot x)^-.\]
 That is, the blowing down sequence has a unique limit. See Section 7.

After proving this, we can establish some good properties in the place far away from
the transition part $\{u=v\}$, such as the existence of a cone of monotonicity for $u-v$.
Indeed, now the situation is quite similar to the Gibbons
conjecture for the Allen-Cahn equation, that is, when $x_n\to\pm\infty$,
we have some uniform convergence of $(u,v)$. (However, we need to note that here $u$ and
$v$ are unbounded and the limit as $x_n\to\pm\infty$ is infinity.)
In Section 8, we use the sliding method by adapting the argument of Farina
\cite{F} ( see also Berestycki-Hamel-Monneau
\cite{BHM}) to prove the existence of a cone of monotonicity for
$(u,v)$. That is, for every unit vector $\tau\in C(e_n,3/4)=\{\tau:
\tau\cdot e_n\geq 3/4\}$,
\[\tau\cdot\nabla u\geq 0,\ \ \ \ \tau\cdot\nabla v\leq 0~~\text{in}~~\mathbb{R}^n.\]
 The main idea is to
propagate the good properties in the part far away from the
transition part $\{u=v\}$ to the part near $\{u=v\}$. Similar ideas
are used in Section 9 to enlarge the cone of monotonicity to a half
space $C(e_n,0)$, which then implies our main result Theorem
\ref{main result}.

In the appendix we give a proof of a local interior version of the
uniform H\"{o}lder estimate from Noris-Tavares-Terracini-Verzini \cite{NTTV}.

In this paper we say a constant $C$ is universal if it is
independent of the base point $x\in\mathbb{R}^n$ and the radius $R$. (In some cases it depends on the solution itself.)
It may be different from line to line.

\section{Some preliminary results}
\numberwithin{equation}{section}
 \setcounter{equation}{0}

In this section we recall some monotonicity formulas for solutions
to \eqref{equation}. Then we will list some useful results, which
will be used many times throughout this paper.
\begin{prop}\label{monotonocity 1}
For $r>0$ and $x\in\mathbb{R}^n$,
$$D(r;x):=r^{2-n}\int_{B_r(x)}|\nabla u|^2+|\nabla v|^2+u^2v^2$$
is nondecreasing in $r$.
\end{prop}
For a proof, see \cite[Lemma 2.1]{C-L 2}. In fact we have
\begin{equation}\label{3.1}
\frac{d}{dr}D(r;x)=2r^{2-n}\left(\int_{\partial
B_r(x)}u_r^2+v_r^2\right)
 +2r^{1-n}\int_{B_r(x)}u^2v^2.
\end{equation}
Next, define
$$H(r;x):=r^{1-n}\int_{\partial B_r(x)}u^2+v^2.$$
By noting that
\begin{equation}\label{3.2}
\frac{d}{dr}H(r;x)=2r^{1-n}\int_{B_r(x)}|\nabla u|^2+|\nabla
v|^2+2u^2v^2,
\end{equation}
we can prove the following (see for example \cite[Proposition 5.2]{BTWW})
\begin{prop}\label{monotonocity}
{\bf (Almgren monotonicity formula.)} For $r>0$ and
$x\in\mathbb{R}^n$,
$$N(r;x,u,v):=\frac{r\int_{B_r(x)}|\nabla u|^2+|\nabla v|^2+u^2v^2}{\int_{\partial B_r(x)}u^2+v^2}$$
is nondecreasing in $r$.
\end{prop}
In this paper, we often omit $u,v$ in $N(r;x,u,v)$ if no ambiguity
occurs. We also denote it by $N(r)$ if $x=0$.

\begin{prop}\label{monotonocity 2}
If $N(r_0;x)\geq d$, then for $r>r_0$,
$$r^{1-n-2d}\int_{\partial B_r(x)}u^2+v^2$$
is nondecreasing in $r$.
\end{prop}
\begin{proof}
Direct calculation using \eqref{3.2} shows
\begin{eqnarray*}
&&\frac{d}{dr}\left(r^{1-n-2d}\int_{\partial B_r(x)}u^2+v^2\right)\\
&=&
-2dr^{-n-2d}\left(\int_{\partial B_r(x)}u^2+v^2\right)+2r^{1-n-2d}\left(\int_{B_r(x)}|\nabla u|^2+|\nabla v|^2+2u^2v^2\right)\\
&\geq& 0.
\end{eqnarray*}
Here we have used Proposition \ref{monotonocity}, in particular, the
fact that $N(r)\geq d$ for every $r\geq r_0$.
\end{proof}
The following result gives a doubling property of $(u,v)$, which is Proposition 5.3 in \cite{BTWW}.
\begin{prop}\label{doubling property}
Let $R>1$ and let $(u,v)$ be a solution of \eqref{equation} on
$B_R$. If $N(R)\leq d$, then for every $1<r_1\leq r_2\leq R$
\begin{equation}\label{eq:h_monotone}
\dfrac{H(r_2)}{H(r_1)}\leq e^{d}\dfrac{r_2^{2d}}{r_1^{2d}}
\end{equation}
\end{prop}
All of these results hold for solutions of \eqref{equation scaled},
if we have the corresponding bounds on
\[N_\kappa(r;x,u_\kappa,v_\kappa):=\frac{r\int_{B_r(x)}|\nabla u_\kappa|^2+|\nabla v_\kappa|^2+\kappa u_\kappa^2v_\kappa^2}{\int_{\partial B_r(x)}u_\kappa^2+v_\kappa^2}.\]
Note that if $(u_\kappa,v_\kappa)$ is a solution of \eqref{equation scaled}, $N_\kappa(r;x,u_\kappa,v_\kappa)$ is still monotone in $r$.

Next we list three useful results. The first one is Lemma 4.4 in
\cite{C-T-V 3}.
\begin{lem}\label{estimate exponential decay}
If in the ball $B_{2R}(0)$, $u\in C^2$ satisfies
\begin{equation} \label{1.1.1}
\left\{ \begin{aligned}
&\Delta u\geq M u , \\
&u\geq 0, \\
&u\leq A ,
                          \end{aligned} \right.
\end{equation}
then
$$\sup_{B_R(0)}u\leq C_{1}(n)Ae^{-C_{2}(n)RM^{\frac{1}{2}}},$$
where $C_{1}(n)$ and $C_{2}(n)$ depend on the dimension $n$ only.
\end{lem}
The next one is an interior version of the uniform H\"{o}lder
estimate in \cite{NTTV}.
\begin{thm}\label{lem uniform Holder}
Let $(u_\kappa, v_\kappa)$ be a sequence of solutions of
\eqref{equation scaled} in $B_{2}(0)$. Assume that as $\kappa\to+\infty$,
$u_\kappa$ and $v_\kappa$ are uniformly bounded, then for any
$\alpha\in(0,1)$, $u_\kappa$ and $v_\kappa$ are uniformly bounded in
$C^\alpha(B_1(0))$.
\end{thm}
We will give a proof in the appendix, following the blow up method
in \cite{NTTV}.

Finally we give a result about the limit of solutions of
\eqref{equation scaled}. This is the consequence of a combination of
the main results in \cite{DWZ2011} and \cite{TT2011}.
\begin{thm}\label{lem 2.3}
 For every $h>0$, there exists a $K>0$, if
$(u_\kappa,v_\kappa)$ is a solution of \eqref{equation scaled} in
$B_2(0)$, with the parameter $\kappa\geq K$, and satisfies
\begin{enumerate}
\item $u_\kappa(0)=v_\kappa(0);$
\item $\int_{\partial B_1(0)}u_\kappa^2+v_\kappa^2=\int_{\partial
B_1(0)}x_1^2;$
\item \[N(2;0,u_\kappa,v_\kappa)=\frac{2\int_{B_2(0)}|\nabla u_\kappa|^2+|\nabla v_\kappa|^2+\kappa u_\kappa^2v_\kappa^2}
{\int_{\partial B_2(0)}u_\kappa^2+v_\kappa^2}\leq 1,\]
\end{enumerate}
then there exists a unit vector $e$ such that
\[\sup_{B_1(0)}\left(|u_\kappa-(e\cdot x)^+|+|v_\kappa-(e\cdot x)^-|\right)\leq h.\]
\end{thm}
\begin{proof}
Assume that there exists $h>0$ and a sub-sequence of $(u_\kappa,v_\kappa)$ satisfying
the assumptions but for every unit vector $e$,
\begin{equation}\label{2.10}
\sup_{B_1(0)}\left(|u_\kappa-(e\cdot x)^+|+|v_\kappa-(e\cdot x)^-|\right)\geq h.
\end{equation}

 By Proposition \ref{doubling property}, there is a universal constant $C$
such that
\[\int_{\partial B_2(0)}u_\kappa^2+v_\kappa^2\leq C.\]
 Because $u_\kappa$ and
$v_\kappa$ are subharmonic functions, they are uniformly bounded in
$B_{3/2}(0)$. By Theorem \ref{lem uniform Holder}, they are also
uniformly bounded in $C^{\alpha}(B_{4/3}(0))$ for every
$\alpha\in(0,1)$. Choosing a subsequence of $(u_\kappa,v_\kappa)$
such that they converge to $(u_\infty,v_\infty)$ uniformly in
$B_1(0)$.
\par
By the main result of \cite{DWZ2011} and \cite{TT2011},
$u_\infty-v_\infty$ is a harmonic function. They satisfy
\begin{enumerate}
\item $u_\infty v_\infty\equiv 0$;
\item $u_\infty(0)=v_\infty(0);$
\item $\int_{\partial B_1(0)}u_\infty^2+v_\infty^2=\int_{\partial
B_1(0)}x_1^2;$
\item \[N_{\infty}(1;0,u_\infty,v_\infty)=\frac{\int_{B_1(0)}|\nabla u_\infty|^2+|\nabla v_\infty|^2}
{\int_{\partial B_1(0)}u_\infty^2+v_\infty^2}=\lim\limits_{\kappa
\to+\infty} N(1;0,u_\kappa,v_\kappa)\leq 1.\]
\end{enumerate}
Note that $N_{\infty}(r;0,u_\infty,v_\infty)$ is exactly the Almgren monotonicity quantity for the harmonic function $u_\infty-v_\infty$. We also have $u_{\infty}(0)=v_{\infty}(0)=0$. By the Almgren
monotonicity formula for harmonic functions, for any $r\in(0,1)$
\begin{equation*}
N_{\infty}(r;0,u_\infty,v_\infty)\geq\lim\limits_{r\to
0}N_{\infty}(r;0,u_\infty,v_\infty)=\text{ord}(u_\infty-v_\infty,0)\geq
1.
\end{equation*}
In the above, $\text{ord}(u_\infty-v_\infty,0)$ is the vanishing order of the
harmonic function $u_\infty-v_\infty$ at $0$. By the Almgren
monotonicity formula for harmonic functions,
$N_{\infty}(r;0,u_\infty,v_\infty)=1$ for every $r\in(0,1)$. This then
implies that $u_\infty-v_\infty=e\cdot x$ for some unit vector $e$. (This
characterization is well known. For a proof and some generalizations
see \cite[Proposition 3.9]{NTTV}.) This is a contradiction, so the
assumption \eqref{2.10} does not hold.
\end{proof}

\section{The blowing down sequence}
\numberwithin{equation}{section}
 \setcounter{equation}{0}
In this section, $(u,v)$ denotes a solution to \eqref{equation}
satisfying the linear growth condition \eqref{linear growth}.
\par
For every $R\geq 1$, denote
$$L(R)^2:=R^{1-n}\int_{\partial B_R(0)}u^2+v^2,$$
and define the blowing down sequence
$$u^R(x):=\frac{1}{L(R)}u(Rx),\ \ \ \ v^R(x):=\frac{1}{L(R)}v(Rx)
.$$
\begin{rmk}\label{growth bound}
By \eqref{linear growth}, there exists a constant $C>0$ such that
$$L(R)\leq CR.$$
On the other hand, by the Liouville type result
(c.f. \cite[Proposition 2.7]{NTTV}), $\forall \alpha\in(0,1)$, $\exists C_{\alpha}>0$
such that
$$L(R)\geq C_{\alpha}R^{\alpha}.$$
\end{rmk}
By the definition, we have the normalized condition
\begin{equation}
\int_{\partial B_1(0)}(u^R)^2+(v^R)^2=1.
\end{equation}
$u^R$ and $v^R$ satisfy \eqref{equation scaled} with
$\kappa(R)=L(R)^2R^2$. Note that as $R\to+\infty$, $\kappa(R)\to+\infty$.
\begin{lem}
There exists a sequence of $R_j\rightarrow+\infty$, such that
$L(\frac{R_j}{2})\geq\frac{1}{3}L(R_j)$.
\end{lem}
\begin{proof}
Assume by the contrary, $\exists R_0>0$ such that, for any $R\geq
R_0$,
\begin{equation*}
L(R)\leq \frac{1}{3}L(2R).
\end{equation*}
An iteration implies, $\forall k>0$,
\begin{equation*}
L(2^kR_0)\geq 3^kL(R_0).
\end{equation*}
On the other hand, by \eqref{linear growth}, we also have
$$L(2^kR_0)\leq C2^k.$$
For $k$ large, this is a contradiction.
\end{proof}
As a consequence, if we choose $R=R_j$ in the definition of $u^R$
and $v^R$ (simply denoted as $u_j,v_j$, and $\kappa_j=\kappa(R_j)$),
we have
\begin{equation}\label{nondegeneracy}
\int_{\partial B_{\frac{1}{2}}(0)}u_j^2+v_j^2\geq \frac{1}{3}.
\end{equation}
Because $u_j$ and $v_j$ are subharmonic, we can get a uniform upper
bound of $u_j$ and $v_j$ on any compact set of $B_1(0)$. Then by
Theorem \ref{lem uniform Holder}, $u_j$ and $v_j$ are uniformly
bounded in $C^{\alpha}(B_{1-\varepsilon}(0))$ for every
$\alpha,\varepsilon\in(0,1)$. Hence we can extract a subsequence of
$j$ (still denoted by $j$), such that $(u_j,v_j)$ converges to
$(u_{\infty},v_{\infty})$ uniformly on any compact set of $B_1(0)$.
By the main result of \cite{DWZ2011} and \cite{TT2011},
$(u_{\infty},v_{\infty})$ satisfies
\begin{equation*}
\Delta(u_{\infty}-v_{\infty})=0.
\end{equation*}
\eqref{nondegeneracy} can be passed to the limit, which implies that
$(u_{\infty},v_{\infty})$ is nonzero. In fact, because
$$u_{\infty}(0)+v_{\infty}(0)=\lim\limits_{j\rightarrow+\infty}u_j(0)+v_j(0)=
\lim\limits_{j\rightarrow+\infty}\frac{u(0)+v(0)}{L_j}=0,$$ both
$u_{\infty}$ and $v_{\infty}$ are nonzero. \footnote { Otherwise,
for example, if $v_{\infty}\equiv 0$, then $u_{\infty}$ is
nonnegative and harmonic. Since $u_\infty(0)=0$, by the strong
maximum principle, $u_\infty\equiv 0$. This means $(u_\infty,v_\infty)=0$, which is a contradiction.}
\begin{prop}\label{prop 4.3}
$u_{\infty}-v_{\infty}$ is linear. That is, there is a constant
$c(n)>0$ depending on the dimension $n$ only, such that, under
suitable coordinates,
$$u_{\infty}=c(n)x_1^+,\ \ \ \ v_{\infty}=c(n)x_1^-.$$
\end{prop}
\begin{proof}
By rescaling the monotonicity formula in Proposition \ref{monotonocity}, we get a similar one for
$(u_j,v_j)$. That is, for $r>0$ and $B_r(x)\subset B_1(0)$,
\begin{equation*}
\frac{r\int_{B_r(x)}|\nabla u_j|^2+|\nabla v_j|^2+\kappa_ju_j^2v_j^2}
{\int_{\partial B_r(x)}u_j^2+v_j^2}
\end{equation*}
is nondecreasing in $r$. We can prove that (cf.
\cite[Theorem 1.4]{NTTV} and \cite[Corollary 2.4]{C-L 2})
\[u_j\rightarrow u_\infty,\ \ v_j\rightarrow v_\infty~~\text{strongly in}~~H^1_{loc}(B_1(0)),\]
$$\kappa_ju_j^2v_j^2\rightarrow 0,~~\mbox{in}~~L^1_{loc}(B_1(0)).$$
Thus for every $B_r(x)\subset\subset B_1(0)$, we have
\begin{eqnarray*}
N_{\infty}(r;x,u_\infty,v_\infty):=\frac{r\int_{B_r(x)}|\nabla
u_{\infty}|^2+|\nabla v_{\infty}|^2} {\int_{\partial
B_r(x)}u_\infty^2+v_\infty^2}=\lim\limits_{j\rightarrow+\infty}\frac{r\int_{B_r(x)}|\nabla
u_j|^2+|\nabla v_j|^2+\kappa_ju_j^2v_j^2} {\int_{\partial
B_r(x)}u_j^2+v_j^2}
\end{eqnarray*}
is nondecreasing in $r$, too. Of course, this fact is nothing else
but the Almgren monotonicity formula for harmonic functions.
\par
Because $u_{\infty}(0)=v_{\infty}(0)=0$, for any $r>0$
\begin{equation*}
N_{\infty}(r;0,u_\infty,v_\infty)\geq\lim\limits_{r\to
0}N_{\infty}(r;0,u_\infty,v_\infty)=\text{ord}(u_\infty-v_\infty,0)\geq
1.
\end{equation*}
Here $\text{ord}(u_\infty-v_\infty,0)$ still denotes the vanishing order of the
harmonic function $u_\infty-v_\infty$ at $0$. We claim that
$N_{\infty}(r;0,u_\infty,v_\infty)=1$ for every $r\in(0,1)$. If
this is true, then by the characterization of linear functions using
the Almgren monotonicity formula (see for example \cite[Proposition 3.9]{NTTV}), we can finish the proof of this proposition.
\par
Assume by the contrary, $\exists r_0\in(0,1)$, such that $N_{\infty}(r_0)=1+2\delta$ with $\delta>0$.
Then for $j$ large,
\begin{equation*}
\frac{r_0\int_{B_{r_0}(0)}|\nabla u_j|^2+|\nabla v_j|^2+\kappa_ju_j^2v_j^2}
{\int_{\partial B_{r_0}(0)}u_j^2+v_j^2}\geq 1+\delta.
\end{equation*}
Rescaling back and using Proposition \ref{monotonocity}, we know
that for $r$ large enough,
\begin{equation*}
\frac{r\int_{B_{r}(0)}|\nabla u|^2+|\nabla v|^2+u^2v^2}
{\int_{\partial B_{r}(0)}u^2+v^2}\geq 1+\delta.
\end{equation*}
By Proposition \ref{monotonocity 2}, for $r>0$ large enough
$$r^{1-n-2(1+\delta)}\int_{\partial B_r(0)}u^2+v^2$$
is nondecreasing in $r$. This contradicts the linear growth
condition \eqref{linear growth} and proves the claim.
\end{proof}

The above blowing down procedure can be performed at any base point
$x\in\mathbb{R}^n$. Thus we get
\begin{coro}\label{coro 4.3}
Let $(u,v)$ be a solution of \eqref{equation} satisfying the linear
growth condition \eqref{linear growth}. Then for every $x\in
\mathbb{R}^n$ and $r>0$, $N(r;x)\leq 1$.
\end{coro}

\begin{rmk}
The blowing down analysis in this section can be preformed for
solutions of \eqref{equation} with polynomial growth. In fact, we
can show that any solution with polynomial growth satisfies
\[\lim\limits_{r\to+\infty}N(r)<+\infty.\]
The blowing down analysis for solutions satisfying this condition
was given in \cite{BTWW}.
\end{rmk}

\section{An Alt-Caffarelli-Friedman monotonicity formula}
\numberwithin{equation}{section}
 \setcounter{equation}{0}
In \cite{NTTV}, Terracini et. al. proved an
Alt-Caffarelli-Friedman type monotonicity for solutions of
\eqref{equation}. In this section we improve their result and
establish a sharp form of the Alt-Caffarelli-Friedman type
monotonicity formula (see Theorem \ref{thm ACF monotonicity} below). In this section we still use $(u,v)$ to denote a solution
of \eqref{equation} on $\mathbb{R}^n$ with linear growth at infinity.

\subsection{The proof}
Before going into the proof, we introduce a useful tool. Fix the
polar coordinates on $\mathbb{S}^{n-1}$ as
$(\cos\alpha,\sin\alpha\cos\alpha_2,\sin\alpha_1\sin\alpha\cdots\sin\alpha_{n-1})$
for $\alpha\in[0,\pi]$ and $\alpha_i\in(0,2\pi)$, $2\leq i\leq n-1$.
Let $\bar{u},\bar{v}$ be two nonnegative
 functions in $L^1(\mathbb{S}^{n-1})$, their rearrangements are two
functions $\bar{u}^*,\bar{v}^*$ satisfying
\begin{enumerate}
\item they depend on $\alpha$ only;
\item $\bar{u}^*$ is non-increasing in $\alpha$ and $\bar{v}^*$ is non-decreasing in
$\alpha$;
\item for every $t>0,
|\{\bar{u}>t\}|=|\{\bar{u}^*>t\}|,|\{\bar{v}>t\}|=|\{\bar{v}^*>t\}|$.
Here $|A|$ denotes the area measure of $A\subset\mathbb{S}^{n-1}$.
\end{enumerate}
Note that this is only the Steiner symmetrization rearrangement (see \cite{LL} for more details), adapted to our special setting
of two component functions $(\bar{u},\bar{v})$ defined on the unit sphere $\mathbb{S}^{n-1}$.

We know if $\bar{u},\bar{v}\in H^1(\mathbb{S}^{n-1})$, then
$\bar{u}^*,\bar{v}^*\in H^1(\mathbb{S}^{n-1})$ (see \cite{F-H}) and
\begin{equation}\label{5.2}
\int_{\mathbb{S}^{n-1}}|\nabla_\theta\bar{u}^*|^2\leq\int_{\mathbb{S}^{n-1}}|\nabla_\theta\bar{u}|^2,\
\ \ \
\int_{\mathbb{S}^{n-1}}|\nabla_\theta\bar{v}^*|^2\leq\int_{\mathbb{S}^{n-1}}|\nabla_\theta\bar{v}|^2.
\end{equation}

Note that in our definition we make these two functions as separated
as possible. More precisely, we have ( similar to the rearrangement inequality \cite[Theorem 3.4]{LL})
\begin{lem}\label{lem 5.2}
If $\bar{u},\bar{v}\in L^2(\mathbb{S}^{n-1})$ are nonnegative and
$\bar{u}^*,\bar{v}^*$ are defined as above, then
\[\int_{\mathbb{S}^{n-1}}\bar{u}^*\bar{v}^*\leq\int_{\mathbb{S}^{n-1}}\bar{u}\bar{v}.\]
\end{lem}
\begin{proof}
By the Fubini theorem,
\[\int_{\mathbb{S}^{n-1}}\bar{u}\bar{v}=\int_0^{+\infty}\int_0^{+\infty}|\{\bar{u}>t\}\cap\{\bar{v}>s\}|dtds.\]
So we only need to prove for every $t,s\in(0,+\infty)$,
\[|\{\bar{u}>t\}\cap\{\bar{v}>s\}|\geq|\{\bar{u}^*>t\}\cap\{\bar{v}^*>s\}|.\]
That is, for every two measurable sets $A,B\subset\mathbb{S}^{n-1}$,
let $A^*,B^*$ be the rearrangement defined above. Then
\[|A\cap B|\geq|A^*\cap B^*|.\]
First we note that this inequality is trivial if $|A^*\cap B^*|=0$.
Next if $|A^*\cap B^*|>0$, by noting that $A^*$ and $B^*$ are
spherical caps with opposite centers, we must have $A^*\cup
B^*=\mathbb{S}^{n-1}$ and trivially $|A\cup B|\leq|A^*\cup
B^*|$. Combining this with the definition of the rearrangement we get
\[|A\cap B|=|A|+|B|-|A\cup B|\geq|A^*|+|B^*|-|A^*\cup B^*|=|A^*\cap B^*|.\]
This finishes the proof.
\end{proof}

In the following we denote, for $x\geq 0$,
\[\gamma(x)=\sqrt{\left(\frac{n-2}{2}\right)^2+x}-\frac{n-2}{2}.\]
\begin{lem}
For every $\lambda\geq 1$, there exists a positive constant $C$,
which depends only on the dimension $n$ and $\lambda$, such that for
every $\kappa\geq 1$ and $\bar{u}_\kappa, \bar{v}_\kappa\in
H^1(\mathbb{S}^{n-1})$ satisfying
\[\int_{\mathbb{S}^{n-1}}\bar{u}_\kappa^2=1,\ \ \ \ \int_{\mathbb{S}^{n-1}}\bar{v}_\kappa^2=\lambda_{\kappa}^2,\]
where $\frac{1}{\lambda}\leq\lambda_{\kappa}\leq\lambda$, then
\[\gamma\left(\frac{\int_{\mathbb{S}^{n-1}}|\nabla_\theta \bar{u}_\kappa|^2+\kappa\bar{u}_\kappa^2\bar{v}_\kappa^2}
{\int_{\mathbb{S}^{n-1}}\bar{u}_\kappa^2}\right)+
\gamma\left(\frac{\int_{\mathbb{S}^{n-1}}|\nabla_\theta
\bar{v}_\kappa|^2+\kappa\bar{u}_\kappa^2\bar{v}_\kappa^2}
{\int_{\mathbb{S}^{n-1}}\bar{v}_\kappa^2}\right)\geq
2-C\kappa^{-1/4}.\]
\end{lem}
\begin{proof}
We divide the proof into three steps. The second step is not so
necessary for the proof. We include it here only to make the picture more
clearer.

{\bf Step 1.} We will consider the constrained minimization problem
\[\min\gamma\left(\frac{\int_{\mathbb{S}^{n-1}}|\nabla_\theta \bar{u}|^2+\kappa\bar{u}^2\bar{v}^2}
{\int_{\mathbb{S}^{n-1}}\bar{u}^2}\right)+
\gamma\left(\frac{\int_{\mathbb{S}^{n-1}}|\nabla_\theta
\bar{v}|^2+\kappa\bar{u}^2\bar{v}^2}
{\int_{\mathbb{S}^{n-1}}\bar{v}^2}\right),\] for $\bar{u},
\bar{v}\in H^1(\mathbb{S}^{n-1})$ satisfying
\[\int_{\mathbb{S}^{n-1}}\bar{u}^2=1,\ \ \ \
\int_{\mathbb{S}^{n-1}}\bar{v}^2=\lambda_{\kappa}^2.\]

After replacing $\bar{v}$ by $\lambda_{\kappa}^{-1}\bar{v}$, we are
led to consider a new minimization problem
\[\min\gamma\left(\int_{\mathbb{S}^{n-1}}|\nabla_\theta \bar{u}|^2+\kappa\lambda_{\kappa}^2\bar{u}^2\bar{v}^2\right)+
\gamma\left(\int_{\mathbb{S}^{n-1}}|\nabla_\theta
\bar{v}|^2+\kappa\bar{u}^2\bar{v}^2\right),\] for $\bar{u},
\bar{v}\in H^1(\mathbb{S}^{n-1})$ satisfying
\[\int_{\mathbb{S}^{n-1}}\bar{u}^2=\int_{\mathbb{S}^{n-1}}\bar{v}^2=1.\]
The direct method shows that the minimizer to this  minimization
problem, $(\bar{u}_\kappa,\bar{v}_\kappa)$, exists. Denote
\[x_\kappa=\int_{\mathbb{S}^{n-1}}|\nabla_\theta \bar{u}_\kappa|^2
+\kappa\lambda_{\kappa}^2\bar{u}_\kappa^2\bar{v}_\kappa^2,~~~
y_\kappa=\int_{\mathbb{S}^{n-1}}|\nabla_\theta
\bar{v}_\kappa|^2+\kappa\bar{u}_\kappa^2\bar{v}_\kappa^2.\] There
exist two Lagrange multipliers $\lambda_{1,\kappa}$ and
$\lambda_{2,\kappa}$ such that
\begin{equation*}
\left\{ \begin{aligned}
 &-\Delta_\theta \bar{u}_\kappa
 +\left(\lambda_{\kappa}^2+\frac{\gamma^{\prime}(y_\kappa)}{\gamma^\prime(x_\kappa)}\right)\kappa \bar{u}_\kappa\bar{v}_\kappa^2
 =\frac{\lambda_{1,\kappa}}{\gamma^\prime(x_\kappa)}\bar{u}_\kappa, \\
 &-\Delta_\theta \bar{v}_\kappa
 +\left(1+\frac{\lambda_{\kappa}^2\gamma^{\prime}(x_\kappa)}{\gamma^\prime(y_\kappa)}\right)\kappa \bar{v}_\kappa\bar{u}_\kappa^2
 =\frac{\lambda_{2,\kappa}}{\gamma^\prime(y_\kappa)}\bar{v}_\kappa.
                          \end{aligned} \right.
\end{equation*}
Moreover, by Lemma \ref{lem 5.2}, we can assume that
$\bar{u}_\kappa$ and $\bar{v}_\kappa$ depend only on $\alpha$, and
one is non-increasing in $\alpha$, the other one non-decreasing in
$\alpha$.
\par
By choosing a test function of the form $(\phi^+,\phi^-)$, where
$\phi$ is an eigenfunction of $-\Delta_\theta$ corresponding to the
eigenvalue $n-1$, with $\int_{\mathbb{S}^{n-1}}\phi^2=2$, we
can get the bound
\begin{equation}\label{5.02}
\gamma(x_\kappa)+\gamma(y_\kappa)\leq 2.
\end{equation}
 Hence $\bar{u}_\kappa$ and $\bar{v}_\kappa$
are uniformly bounded in $H^1(\mathbb{S}^{n-1})$. Multiplying the
equations by $\bar{u}_\kappa$ and $\bar{v}_\kappa$ respectively and
integrating by parts, we get
\begin{equation}\label{5.3}
\left\{ \begin{aligned}
 &x_\kappa+\frac{\gamma^\prime(y_\kappa)}{\gamma^\prime(x_\kappa)}
 \int_{\mathbb{S}^{n-1}}\kappa\bar{u}_\kappa^2\bar{v}_\kappa^2
 =\frac{\lambda_{1,\kappa}}{\gamma^\prime(x_\kappa)}, \\
 &y_\kappa+\frac{\lambda_{\kappa}^2\gamma^\prime(x_\kappa)}{\gamma^\prime(y_\kappa)}
 \int_{\mathbb{S}^{n-1}}\kappa\bar{u}_\kappa^2\bar{v}_\kappa^2
 =\frac{\lambda_{2,\kappa}}{\gamma^\prime(y_\kappa)}.
                          \end{aligned} \right.
\end{equation}
In particular, $\lambda_{1,\kappa}$ and $\lambda_{2,\kappa}$ are
positive and uniformly bounded.
\par
{\bf Step 2.} Choosing
\[\xi_\kappa=\sqrt{\frac{\lambda_{\kappa}^2+\frac{\gamma^{\prime}(y_\kappa)}{\gamma^\prime(x_\kappa)}}
{1+\frac{\lambda_{\kappa}^2\gamma^{\prime}(x_\kappa)}{\gamma^\prime(y_\kappa)}}}\]
and defining
\[\hat{u}_\kappa:=\bar{u}_\kappa,\ \ \ \ \hat{v}_\kappa:=\xi_\kappa\bar{v}_\kappa,\]
we have
\begin{equation}\label{an equation on sphere}
\left\{ \begin{aligned}
 &-\Delta_\theta \hat{u}_\kappa
 +\hat{\kappa} \hat{u}_\kappa\hat{v}_\kappa^2
 =\hat{\lambda}_{1,\kappa}\hat{u}_\kappa, \\
 &-\Delta_\theta \hat{v}_\kappa
 +\hat{\kappa} \hat{v}_\kappa\hat{u}_\kappa^2
 =\hat{\lambda}_{2,\kappa}\hat{v}_\kappa.
                          \end{aligned} \right.
\end{equation}
Here
\[\hat{\kappa}=\kappa\left(1+\frac{\lambda_{\kappa}^2\gamma^{\prime}(x_\kappa)}{\gamma^\prime(y_\kappa)}\right),
\ \ \
\hat{\lambda}_{1,\kappa}=\frac{\lambda_{1,\kappa}}{\gamma^\prime(x_\kappa)},
\ \ \
\hat{\lambda}_{2,\kappa}=\frac{\lambda_{2,\kappa}}{\gamma^\prime(y_\kappa)}.\]
 As $\kappa\to+\infty$, $\hat{\kappa}\to+\infty$.
 Without loss of generality we can assume that there exist nonnegative
constants $\xi_\infty$, $\hat{\lambda}_{1,\infty}$ and
$\hat{\lambda}_{2,\infty}$ such that
\[\lim\limits_{\kappa\to+\infty}\xi_\kappa=\xi_\infty,\ \ \
\lim\limits_{\kappa\to+\infty}\hat{\lambda}_{1,\kappa}=\hat{\lambda}_{1,\infty},\
\ \
\lim\limits_{\kappa\to+\infty}\hat{\lambda}_{2,\kappa}=\hat{\lambda}_{2,\infty}.\]
We can also get a uniform upper bound of $\hat{u}_\kappa$ and
$\hat{v}_\kappa$ by the Moser iteration, using the uniform
$H^1(\mathbb{S}^{n-1})$ bound on  $\hat{u}_\kappa$ and
$\hat{v}_\kappa$ and the elliptic inequalities
\[-\Delta_\theta \hat{u}_\kappa
 \leq\hat{\lambda}_{1,\kappa}\hat{u}_\kappa, \ \ \ \ -\Delta_\theta \hat{v}_\kappa
 \leq\hat{\lambda}_{2,\kappa}\hat{v}_\kappa.\]
Similar to the proof of Theorem \ref{lem 2.3}, there is a solution $w$ of
\begin{equation}\label{equation for w}
-\Delta_\theta w=\hat{\lambda}_{1,\infty}w^+-\hat{\lambda}_{2,\infty}w^-
\end{equation}
such that $\hat{u}_\kappa\to w^+$, $\hat{v}_\kappa\to w^-$ in
$C(\mathbb{S}^{n-1})\cap H^1(\mathbb{S}^{n-1})$. Moreover,
\[\lim\limits_{\kappa\to+\infty}\int_{\mathbb{S}^{n-1}}\hat{\kappa}\hat{u}_\kappa^2\hat{v}_\kappa^2=0.\]
This implies
\begin{equation}\label{5.4}
\lim\limits_{\kappa\to+\infty}\int_{\mathbb{S}^{n-1}}\kappa\bar{u}_\kappa^2\bar{v}_\kappa^2=0.
\end{equation}
Note that
\[\int_{\mathbb{S}^{n-1}}(w^+)^2=\lim\limits_{\kappa\to+\infty}\int_{\mathbb{S}^{n-1}}\hat{u}_\kappa^2=1,
\ \ \ \
\int_{\mathbb{S}^{n-1}}(w^-)^2=\lim\limits_{\kappa\to+\infty}\int_{\mathbb{S}^{n-1}}\hat{v}_\kappa^2=\xi_\infty^2.\]
Hence, with the help of \eqref{5.4} we get
\begin{eqnarray*}
\lim\limits_{\kappa\to+\infty}\int_{\mathbb{S}^{n-1}}|\nabla_\theta\bar{u}_\kappa|^2
+\kappa\lambda_\kappa^2\bar{u}_\kappa^2\bar{v}_\kappa^2 =
\lim\limits_{\kappa\to+\infty}\int_{\mathbb{S}^{n-1}}|\nabla_\theta\hat{u}_\kappa|^2
+\hat{\kappa}\hat{u}_\kappa^2\hat{v}_\kappa^2
=\frac{\int_{\mathbb{S}^{n-1}}|\nabla_\theta
w^+|^2}{\int_{\mathbb{S}^{n-1}}(w^+)^2}.
\end{eqnarray*}
\begin{eqnarray*}
\lim\limits_{\kappa\to+\infty}\int_{\mathbb{S}^{n-1}}|\nabla_\theta\bar{v}_\kappa|^2
+\kappa\bar{u}_\kappa^2\bar{v}_\kappa^2 =
\lim\limits_{\kappa\to+\infty}\xi_\kappa^{-2}\int_{\mathbb{S}^{n-1}}|\nabla_\theta\hat{v}_\kappa|^2
+\hat{\kappa}\hat{u}_\kappa^2\hat{v}_\kappa^2
=\frac{\int_{\mathbb{S}^{n-1}}|\nabla_\theta
w^-|^2}{\int_{\mathbb{S}^{n-1}}(w^-)^2}.
\end{eqnarray*}
Letting $\kappa\to+\infty$ in \eqref{5.02}, we get
\[\gamma\left(\frac{\int_{\mathbb{S}^{n-1}}|\nabla_\theta
w^+|^2}{\int_{\mathbb{S}^{n-1}}(w^+)^2}\right)+\gamma\left(\frac{\int_{\mathbb{S}^{n-1}}|\nabla_\theta
w^-|^2}{\int_{\mathbb{S}^{n-1}}(w^-)^2}\right)\leq 2.\] By
\cite{ACF} (see also \cite{F-H} and \cite{HHT}),
\begin{equation}\label{min on sphere}
\min_{w\in
H^1(\mathbb{S}^{n-1})}\gamma\left(\frac{\int_{\mathbb{S}^{n-1}}|\nabla_\theta
w^+|^2}{\int_{\mathbb{S}^{n-1}}(w^+)^2}\right)+\gamma\left(\frac{\int_{\mathbb{S}^{n-1}}|\nabla_\theta
w^-|^2}{\int_{\mathbb{S}^{n-1}}(w^-)^2}\right)=2.
\end{equation}
So the above inequality must be an equality. By testing the equation of $w$, \eqref{equation for w}, with $w^+$ and $w^-$
respectively, we have
\[\frac{\int_{\mathbb{S}^{n-1}}|\nabla_\theta
w^+|^2}{\int_{\mathbb{S}^{n-1}}(w^+)^2}=\hat{\lambda}_{1,\infty},
\ \ \ \
\frac{\int_{\mathbb{S}^{n-1}}|\nabla_\theta
w^-|^2}{\int_{\mathbb{S}^{n-1}}(w^-)^2}=\hat{\lambda}_{2,\infty}.\]
Combining all of these we see $\lambda_{1,\infty}=\lambda_{2,\infty}=n-1$.
In particular, $w$ is an
eigenfunction corresponding to the eigenvalue $n-1$. Then by
\eqref{5.3},
\[\lim\limits_{\kappa\to+\infty}x_\kappa=\lim\limits_{\kappa\to+\infty}y_\kappa=n-1.\]
By the definition of $\xi_\kappa$ we get
\[\lim\limits_{\kappa\to+\infty}\xi_\kappa=1.\]
By the convergence of $\hat{u}_\kappa$ and $\hat{v}_\kappa$ in
$C(\mathbb{S}^{n-1})\cap H^1(\mathbb{S}^{n-1})$, $\bar{u}_\kappa\to
w^+$, $\bar{v}_\kappa\to w^-$ in $C(\mathbb{S}^{n-1})\cap
H^1(\mathbb{S}^{n-1})$. By the symmetry of $\bar{u}_\kappa$ and
$\bar{v}_\kappa$, there exists a constant $c(n)$ such that
\[w=c(n)\cos\alpha.\]
\par
{\bf Step 3.} Because $\bar{u}_\kappa$ and $\bar{v}_\kappa$ depend
on $\alpha$ only, they satisfy a system of ordinary
differential equations. By suitably modifying the arguments in
\cite{blwz} (cf. \cite[Theorem 1.1 and Lemma 2.4]{blwz}), we can prove the following

{\bf Claim.}
As $\kappa\to+\infty$, $\hat{u}_\kappa$ and $\hat{v}_\kappa$ (and hence $\bar{u}_\kappa$
and $\bar{v}_\kappa$) are uniformly Lipschitz continuous.
There exists a constant $C$ such that for all $\kappa$,
\begin{equation}\label{5.5}
\bar{u}_\kappa\bar{v}_\kappa\leq C\kappa^{-\frac{1}{2}}.
\end{equation}
{\it Proof.}
By the results obtained in the previous step, $\hat{u}_\kappa\to c(n)(\cos\alpha)^+$ and $\hat{v}_\kappa\to c(n)(\cos\alpha)^-$
uniformly on $\mathbb{S}^{n-1}$. By Lemma \ref{estimate exponential decay}, for any $\delta>0$, in $[0,\pi/2-\delta]$,
$\hat{v}_\kappa$ are exponentially small in $\hat{\kappa}$. Substituting this estimate into the equations of $\hat{u}_\kappa$ and $\hat{v}_\kappa$, we see $\hat{u}_\kappa$ and $\hat{v}_\kappa$ are uniformly Lipschitz continuous in $[0,\pi/2-\delta]$, and \eqref{5.5} holds true in this interval. By changing the positions of $\hat{u}_\kappa$ and $\hat{v}_\kappa$, these also hold true in $[\pi/2+\delta,\pi]$.

Now we can restrict our attention to the interval $[\pi/4,3\pi/4]$. Note that $\hat{u}_\kappa$ and $\hat{v}_\kappa$ satisfy
\begin{equation}\label{an equation on sphere 2}
\left\{ \begin{aligned}
 &\frac{d^2}{d\alpha^2} \hat{u}_\kappa+\left(n-2\right)\cot\alpha\frac{d}{d\alpha} \hat{u}_\kappa
 =\hat{\kappa} \hat{u}_\kappa\hat{v}_\kappa^2-\hat{\lambda}_{1,\kappa}\hat{u}_\kappa, \\
 &\frac{d^2}{d\alpha^2} \hat{v}_\kappa+\left(n-2\right)\cot\alpha\frac{d}{d\alpha} \hat{v}_\kappa
 =\hat{\kappa} \hat{v}_\kappa\hat{u}_\kappa^2-\hat{\lambda}_{2,\kappa}\hat{v}_\kappa.
                          \end{aligned} \right.
\end{equation}
By integrating \eqref{an equation on sphere} on the whole $\mathbb{S}^{n-1}$ and noting that $\hat{\lambda}_{1,\kappa}\hat{u}_\kappa$ and $\hat{\lambda}_{2,\kappa}\hat{v}_\kappa$ are uniformly bounded in $L^1(\mathbb{S}^{n-1})$, we see
$\hat{\kappa} \hat{u}_\kappa\hat{v}_\kappa^2$ and $\hat{\kappa} \hat{v}_\kappa\hat{u}_\kappa^2$ are also uniformly bounded in
$L^1(\mathbb{S}^{n-1})$. Hence they are uniformly bounded in $L^1([\pi/4,3\pi/4])$. Now for any $\alpha\in[\pi/4,3\pi/4]$, integrating the equation of $\hat{u}_\kappa$ in the interval $[\pi/4,\alpha]$, we get
\[\frac{d}{d\alpha} \hat{u}_\kappa(\alpha)-\frac{d}{d\alpha} \hat{u}_\kappa(\pi/4)=\int_{\pi/4}^\alpha\left(-\left(n-2\right)\cot\alpha\frac{d}{d\alpha} \hat{u}_\kappa
+\hat{\kappa} \hat{u}_\kappa\hat{v}_\kappa^2-\hat{\lambda}_{1,\kappa}\hat{u}_\kappa\right)d\alpha.\]
By the above discussion and the uniform bound on $\|\hat{u}_\kappa\|_{H^1(\mathbb{S}^{n-1})}$, the right hand side is uniformly bounded, independent of $\hat{\kappa}\to+\infty$ and $\alpha\in[\pi/4,3\pi/4]$. Since $\frac{d}{d\alpha} \hat{u}_\kappa(\pi/4)$ is also uniformly bounded, this gives the uniform bound of $\frac{d}{d\alpha} \hat{u}_\kappa(\alpha)$ for all $\alpha\in[0,\pi]$. The same method works for $\hat{v}_\kappa$.

To prove \eqref{5.5}, assume by the contrary that there exist
$\alpha_\kappa\in[0,\pi]$ such that
$\kappa^{1/2}\bar{u}_\kappa(\alpha_\kappa)\bar{v}_\kappa(\alpha_\kappa)\to+\infty$.
By Lemma \ref{estimate exponential decay}, both
$\bar{u}_\kappa(\alpha_\kappa)$ and $\bar{v}_\kappa(\alpha_\kappa)$
go to zero. In particular, $\alpha_\kappa$ converges to $\pi/2$. Take
$\varepsilon_\kappa=\bar{u}_\kappa(\alpha_\kappa)+\bar{v}_\kappa(\alpha_\kappa)$
and define
\[\widetilde{u}_\kappa(\alpha)=\varepsilon_\kappa^{-1}\bar{u}_\kappa(\alpha_\kappa+\varepsilon_\kappa\alpha),
\ \  \ \
\widetilde{v}_\kappa(\alpha)=\varepsilon_\kappa^{-1}\bar{v}_\kappa(\alpha_\kappa+\varepsilon_\kappa\alpha)\ \ \ \mbox{for}\ \alpha\in(-\frac{\pi}{4\varepsilon_\kappa},\frac{\pi}{4\varepsilon_\kappa}).\]
Since $\widetilde{u}_\kappa(0)+\widetilde{v}_\kappa(0)=1$, and these
two functions are uniformly Lipschitz, $(\widetilde{u}_\kappa,
\widetilde{v}_\kappa)$ converges uniformly to
$(\widetilde{u}_\infty, \widetilde{v}_\infty)$ on any compact set of
$\mathbb{R}$. Rescaling \eqref{an equation on sphere 2} we have
\[\frac{d^2}{d\alpha^2}\widetilde{u}_\kappa
=\kappa\varepsilon_\kappa^4\widetilde{u}_\kappa\widetilde{v}_\kappa^2+o(1),\ \ \frac{d^2}{d\alpha^2}\widetilde{v}_\kappa
=\kappa\varepsilon_\kappa^4\widetilde{v}_\kappa\widetilde{u}_\kappa^2+o(1),\]
where $o(1)\to0$ uniformly on any compact set of
$\mathbb{R}$. By definition and our assumption,
$\kappa\varepsilon_\kappa^4\geq
\kappa\bar{u}_\kappa(\alpha_\kappa)^2\bar{v}_\kappa(\alpha_\kappa)^2\to+\infty$.
Hence we must have $\widetilde{u}_\infty\widetilde{v}_\infty\equiv
0$. Assume that $\widetilde{u}_\infty(0)=1$. There exists a constant
$\delta>0$ such that, for $\kappa$ large,
$\widetilde{u}_\kappa\geq1/2$ in $[-\delta,\delta]$. By Lemma
\ref{estimate exponential decay}, we have
\[
\kappa^{1/2}\bar{u}_\kappa(\alpha_\kappa)\bar{v}_\kappa(\alpha_\kappa)
=\kappa^{1/2}\varepsilon_\kappa^2\widetilde{u}_\kappa(0)\widetilde{v}_\kappa(0)
\leq
C_1\kappa^{1/2}\varepsilon_\kappa^2e^{-C_2\kappa^{1/2}\varepsilon_\kappa^2},\]
which goes to $0$ because
$\kappa^{1/2}\varepsilon_\kappa^2\to+\infty$. This is a contradiction and finishes the
proof of the claim.

\eqref{5.5} implies that, for every $\theta\in\mathbb{S}^{n-1}$,
either $\bar{u}_\kappa(\theta)$ or $\bar{v}_\kappa(\theta)$ is less
than $C\kappa^{-1/4}$. Then
\begin{equation}\label{5.6}
\int_{\mathbb{S}^{n-1}}\kappa\bar{u}_\kappa^2\bar{v}_\kappa^2\leq
C\kappa^{-1/4}\int_{\mathbb{S}^{n-1}}\kappa\bar{u}_\kappa\bar{v}_\kappa^2
+\kappa\bar{v}_\kappa\bar{u}_\kappa^2\leq C\kappa^{-1/4}.
\end{equation}

Take $f_\kappa=(\bar{u}_\kappa-\bar{v}_\kappa)^+$,
$g_\kappa=(\bar{u}_\kappa-\bar{v}_\kappa)^-$. First by \eqref{5.5}
we have
\begin{equation}\label{5.7}
\int_{\mathbb{S}^{n-1}}|f_\kappa-\bar{u}_\kappa|^2\leq
\int_{\{\bar{u}_\kappa>\bar{v}_\kappa\}}\bar{v}_\kappa^2
+\int_{\{\bar{u}_\kappa<\bar{v}_\kappa\}}\bar{u}_\kappa^2 \leq
C\kappa^{-1/2}.
\end{equation}
The same estimate holds for $g_\kappa-\bar{v}_\kappa$.
\par
Next we estimate
\[\int_{\mathbb{S}^{n-1}}|\nabla_\theta f_\kappa|^2=
\int_{\{\bar{u}_\kappa>\bar{v}_\kappa\}}|\nabla_\theta\bar{u}_\kappa|^2+|\nabla_\theta\bar{v}_\kappa|^2
-2\nabla_\theta\bar{u}_\kappa\cdot\nabla_\theta\bar{v}_\kappa.\] By
the symmetry and monotonicity of $\bar{u}_\kappa$ and
$\bar{v}_\kappa$, $\partial\{\bar{u}_\kappa>\bar{v}_\kappa\}$ is a
smooth hypersurface (i.e. a circle). So by the uniform Lipschitz
continuity of $\bar{u}_\kappa$ and \eqref{5.5}, we get
\begin{eqnarray*}
\big|\int_{\{\bar{u}_\kappa>\bar{v}_\kappa\}}\nabla_\theta\bar{u}_\kappa\cdot\nabla_\theta\bar{v}_\kappa\big|
&=&\big|\int_{\partial\{\bar{u}_\kappa>\bar{v}_\kappa\}}\frac{\partial\bar{u}_\kappa}{\partial\nu}\bar{v}_\kappa
-\int_{\{\bar{u}_\kappa>\bar{v}_\kappa\}}\bar{v}_\kappa\Delta_\theta\bar{u}_\kappa\big|\\
&\leq&\int_{\partial\{\bar{u}_\kappa>\bar{v}_\kappa\}}C\kappa^{-1/4}
+C\kappa^{-1/4}\int_{\{\bar{u}_\kappa>\bar{v}_\kappa\}}\Delta_\theta\bar{u}_\kappa\\
&\leq&C\kappa^{-1/4}.
\end{eqnarray*}
Similarly we get
\[\int_{\{\bar{u}_\kappa>\bar{v}_\kappa\}}|\nabla_\theta\bar{v}_\kappa|^2\leq C\kappa^{-1/4}.\]
Combining these with \eqref{5.6}, we see
\begin{equation}\label{5.8}
\int_{\mathbb{S}^{n-1}}|\nabla_\theta f_\kappa|^2\leq
\int_{\mathbb{S}^{n-1}}|\nabla_\theta\bar{u}_\kappa|^2+\kappa\lambda_\kappa^2\bar{u}_\kappa^2\bar{v}_\kappa^2
+C\kappa^{-1/4}.
\end{equation}
We can also get similar estimates for $g_\kappa$.

By noting \eqref{5.7} and \eqref{min on sphere}, we have
\[2\leq\gamma\left(\frac{\int_{\mathbb{S}^{n-1}}|\nabla_\theta f_\kappa|^2}{\int_{\mathbb{S}^{n-1}}f_\kappa^2}\right)
+\gamma\left(\frac{\int_{\mathbb{S}^{n-1}}|\nabla_\theta
g_\kappa|^2}{\int_{\mathbb{S}^{n-1}}g_\kappa^2}\right)\leq
\gamma(x_\kappa)+\gamma(y_\kappa)+C\kappa^{-1/4}.\] This is nothing
else but a reformulation of the claim in this lemma.
\end{proof}
Now we can state the Alt-Caffarelli-Friedman monotonicity formula.
In the following we denote
\[J(r;x)=r^{-4}\left(\int_{B_r(x)}\frac{|\nabla
u(y)|^2+u(y)^2v(y)^2}{|x-y|^{n-2}}dy\right)\left( \int_{B_r(x)}\frac{|\nabla
v(y)|^2+u(y)^2v(y)^2}{|x-y|^{n-2}}dy\right).\] If $x=0$, we simply write
this as $J(r)$.
\begin{thm}\label{thm ACF monotonicity}
Let $(u,v)$ be a solution of \eqref{equation} satisfying
\eqref{linear growth}. There exists a
positive constant $C$, 
such that for every 
$r>1$,
\[e^{-Cr^{-1/2}}J(r)\] is nondecreasing in $r$.
\end{thm}
\begin{proof}
 As in the proof of Lemma 2.5
in \cite{NTTV}, we have
\begin{equation}\label{5.1}
\frac{d}{dr}\log
J(r)\geq-\frac{4}{r}+\frac{2}{r}\left[\gamma\left(\frac{\int_{\mathbb{S}^{n-1}}|\nabla_\theta
\bar{u}|^2+r^2\bar{u}^2\bar{v}^2}
{\int_{\mathbb{S}^{n-1}}\bar{u}^2}\right)+\gamma\left(\frac{\int_{\mathbb{S}^{n-1}}|\nabla_\theta
\bar{v}|^2+r^2\bar{u}^2\bar{v}^2}
{\int_{\mathbb{S}^{n-1}}\bar{v}^2}\right)\right],
\end{equation}
where
\[\bar{u}(\theta)=u(r\theta),\ \ \ \ \bar{v}(\theta)=v(r\theta).\]
Because $u$ and $v$ are subharmonic, by the mean value inequality we
get
\[\int_{\mathbb{S}^{n-1}}\bar{u}=\frac{1}{|\partial B_r|}\int_{\partial B_r(0)}u\geq
u(0),\ \ \ \int_{\mathbb{S}^{n-1}}\bar{v}=\frac{1}{|\partial
B_r|}\int_{\partial B_r(0)}v\geq v(0).\]
 By Proposition \ref{prop 4.3}, we have
\[\lim\limits_{r\to+\infty}\frac{\int_{\mathbb{S}^{n-1}}\bar{u}^2}{\int_{\mathbb{S}^{n-1}}\bar{v}^2}=1.\]
After a normalization in $L^2(\mathbb{S}^{n-1})$ we can apply the
previous lemma (with a fixed $\lambda\geq1$ for all $r\geq 1$) to
deduce that
\[\gamma\left(\frac{\int_{\mathbb{S}^{n-1}}|\nabla_\theta
\bar{u}|^2+r^2\bar{u}^2\bar{v}^2}
{\int_{\mathbb{S}^{n-1}}\bar{u}^2}\right)+\gamma\left(\frac{\int_{\mathbb{S}^{n-1}}|\nabla_\theta
\bar{v}|^2+r^2\bar{u}^2\bar{v}^2}
{\int_{\mathbb{S}^{n-1}}\bar{v}^2}\right)\geq 2-Cr^{-1/2}.\]
Substituting this into \eqref{5.1} we get for all $r$ large
\[\frac{d}{dr}\log J(r)\geq -Cr^{-1-1/2}.\]
By taking a larger constant $C$, we know for all $r\geq 1$, $\log
J(r)-Cr^{-1/2}$ is nondecreasing in $r$. This finishes the proof.
\end{proof}

\subsection{Some consequences}
\begin{prop}\label{coro 5.3}
For a solution $(u,v)$ of \eqref{equation} satisfying \eqref{linear
growth}, there exists a constant $C$, such that for every $r>1$,
\[\frac{1}{C}\leq J(r)\leq C.\]
Moreover, the limit $\lim\limits_{r\to+\infty}J(r)\in(0,+\infty)$
exists.
\end{prop}
\begin{proof}
The lower bound is guaranteed by the almost monotonicity of $J(r)$
and the fact that $u$ and $v$ are not constants. The upper bound can
be obtained by combining the linear growth condition with the
following estimate
\begin{equation}\label{5.9}
r^{-2}\int_{B_r(0)}\frac{|\nabla
u(y)|^2+u(y)^2v(y)^2}{|y|^{n-2}}dy\leq Cr^{-n-2}\int_{B_{2r}}u(y)^2dy.
\end{equation}
This is because for every $\eta\in C_0^\infty(\mathbb{R}^n)$, we have
\begin{eqnarray*}
\int_{\mathbb{R}^n}\frac{|\nabla u(y)|^2+u(y)^2v(y)^2}{|y|^{n-2}}\eta(y)^2dy
&=&\int_{\mathbb{R}^n}\Delta\frac{u(y)^2}{2}|y|^{2-n}\eta(y)^2dy\\
&=&\int_{\mathbb{R}^n}\frac{u(y)^2}{2}\Delta\left(|y|^{2-n}\eta(y)^2\right)dy\\
&\leq&\int_{\mathbb{R}^n}\frac{u(y)^2}{2}\left(2\nabla|y|^{2-n}\cdot\nabla\eta(y)^2+|y|^{2-n}\Delta\eta(y)^2\right)dy.
\end{eqnarray*}
In the above, $\Delta\left(|y|^{2-n}\eta(y)^2\right)$ is understood in the distributional sense. In particular, to get the last inequality, we used the fact that $|y|^{2-n}$ is a superharmonic function, which implies that
\begin{eqnarray*}
\Delta\left(|y|^{2-n}\eta(y)^2\right)&=&\Delta|y|^{2-n}\eta(y)^2+2\nabla|y|^{2-n}\cdot\nabla\eta(y)^2+|y|^{2-n}\Delta\eta(y)^2\\
&\leq&2\nabla|y|^{2-n}\cdot\nabla\eta(y)^2+|y|^{2-n}\Delta\eta(y)^2.
\end{eqnarray*}
This inequality can also be proved by first cutting a ball $B_\delta(0)$, integrating by parts on $\mathbb{R}^n\setminus B_\delta(0)$, and then taking the limit $\delta\to0$. After the integration by parts, there is a term
\[-\frac{n-2}{2}\delta^{1-n}\int_{\partial B_\delta(0)}u^2\eta^2,\]
which does not converge to $0$ as $\delta\to0$. However, it has a favorable sign, thus can be thrown away in the last step.

By choosing $\eta\equiv1$ in $B_r(0)$, $\eta\equiv0$ outside
$B_{2r}(0)$ and $|\nabla\eta|^2+|\Delta\eta|\leq\frac{16}{r^2}$ we
get \eqref{5.9}.

Finally, it is clear that
\[\lim\limits_{r\to+\infty}J(r)=\lim\limits_{r\to+\infty}e^{-Cr^{-\frac{1}{2}}}J(r)\in(\frac{1}{C},C).\qedhere\]
\end{proof}

This implies a nondegeneracy result.
\begin{coro}\label{coro 5.4}
For a solution $(u,v)$ of \eqref{equation} satisfying \eqref{linear
growth}, there exists a positive constant $C$, such that for all $r>4$,
\[\int_{\partial B_r(0)}u+v\geq \frac{1}{C}r^n.\]
\end{coro}
\begin{proof}
Take an $r>4$ and define $\varepsilon$ so that it satisfies
\[\int_{\partial B_r(0)}u+v=\varepsilon r^n.\]
Because $u$ and $v$ are subharmonic,
\[\sup_{B_{r/2}(0)}\left(u+v\right)\leq C\varepsilon r.\]
Using \eqref{5.9} we get
\[J(r/4)\leq C\varepsilon^2.\]
By Proposition \ref{coro 5.3}, we get $\varepsilon\geq 1/C$,
where $C$ is a constant independent of $r$.
\end{proof}

\begin{rmk}\label{rmk 5.5}
Combining this with the results in Section 3, we know for every
$R_j\to+\infty$, up to a subsequence of $j$, there exists a vector
$e$ such that on any compact set of $\mathbb{R}^n$, there is the
uniform convergence
\[u_j(x)=\frac{1}{R_j}u(R_jx)\to(e\cdot x)^+,\ \ \ v_j(x)=\frac{1}{R_j}v(R_jx)\to(e\cdot x)^-.\]
Note that $J$ is invariant under such a scaling. By Proposition
\ref{prop 4.3},
\begin{eqnarray*}
\lim\limits_{j\to+\infty}J(R_j)&=&\lim\limits_{j\to+\infty}\int_{B_1(0)}\frac{|\nabla
u_j(y)|^2+R_j^4u_j(y)^2v_j(y)^2}{|y|^{n-2}}\int_{B_1(0)}\frac{|\nabla
v_j(y)|^2+R_j^4u_j(y)^2v_j(y)^2}{|y|^{n-2}}\\
&=&c(n)|e|^2.
\end{eqnarray*}
 After a scaling of the form
$(u(x),v(x))\rightarrow(\lambda u(\lambda x),\lambda v(\lambda x))$
for some $\lambda>0$, we can assume $|e|=1$. However, at this stage
we do not know whether such $e$ is unique. It may depend on the
sequence $R_j$.
\end{rmk}

\begin{lem}\label{lem 5.7}
For a solution $(u,v)$ of \eqref{equation} satisfying \eqref{linear
growth}, there exists a universal constant $C$ such that for every $x\in\{
u=v\}$,\[\int_{B_1(x)}u^2\geq C~~\text{and}~~\int_{B_1(x)}v^2\geq
C.\]
\end{lem}
\begin{proof}
Assume by the contrary, there exists a sequence of $x_i\in\{u=v\}$
such that
\[\lim\limits_{i\to+\infty}\int_{B_1(x_i)}u^2=0.\]
We claim that
\[\lim\limits_{i\to+\infty}\int_{B_1(x_i)}v^2=0.\]
Otherwise there exists a $\delta>0$ such that
\[\lim\limits_{i\to+\infty}\int_{B_1(x_i)}v^2\geq\delta^2.\]
Consider
\[u_i(x)=\frac{1}{L_i}u(x_i+x),\ \ \ \ v_i(x)=\frac{1}{L_i}v(x_i+x)\]
where $L_i\geq \delta$ is chosen so that
\begin{equation}\label{5.01}
\int_{B_1(0)}u_i^2+v_i^2=1.
\end{equation}
$(u_i,v_i)$ satisfies \eqref{equation scaled} with
$\kappa=L_i^2\geq\delta^2$. By scaling the doubling property
Proposition \ref{doubling property}, we obtain
\[\int_{B_2(0)}u_i^2+v_i^2\leq 64\cdot 2^n.\]
Exactly as in Section 3, we know $u_i$ and $v_i$ are uniformly
bounded and uniformly $1/2$-H\"{o}lder continuous in $B_{3/2}(0)$.
Assume their limits (of a subsequence) are $u_\infty$ and
$v_{\infty}$. Then
\[\int_{B_1(0)}u_\infty^2=0,\ \ \ \ \int_{B_1(0)}v_\infty^2=1.\]
So $u_\infty\equiv 0$.

Since $x_i\in\{u=v\}$,
\[u_\infty(0)-v_\infty(0)=\lim\limits_{i\to+\infty}u_i(0)-v_i(0)=0.\]
Hence $v_\infty(0)=0$. Because $v_\infty$ is a nonnegative harmonic
function, \footnote{There are two cases: if $L_i$ are uniformly
bounded, this comes directly from the second equation in
\eqref{equation scaled} and the fact that $u_\infty\equiv 0$, and if
$L_i\to+\infty$, this can be deduced from Theorem \ref{lem 2.3}.} by
the strong maximum principle, $v_\infty\equiv 0$. This contradicts
\eqref{5.01} and we prove the claim.
\par
Now we use Corollary \ref{coro 5.4} and Proposition \ref{doubling
property} to get, for $r\gg|x_i|$
\begin{eqnarray*}
Cr^{n+2}&\leq&\int_{B_r(0)}u^2+v^2\\
&\leq&\int_{B_{r+|x_i|}(x_i)}u^2+v^2\\
&\leq&Cr^{n+2}\int_{B_1(x_i)}u^2+v^2=o(r^{n+2}).
\end{eqnarray*}
Thus for $i$ large we get a contradiction. That is, our assumption
at the beginning of the proof is not true.
\end{proof}
Repeating the proof of Theorem \ref{thm ACF monotonicity}
we get
\begin{coro}\label{coro 5.8}
For a solution $(u,v)$ of \eqref{equation} satisfying \eqref{linear
growth}, there is a universal constant $C$, such that for every $x\in\{u=v\}$
and $r>1$, $e^{-Cr^{-1/2}}J(r;x)$ is nondecreasing in $r$.
\end{coro}
\begin{proof}
We only need to show that, there exists a constant $\lambda\geq 1$,
such that for all $x\in\{u=v\}$ and $r\geq1$,
\[\frac{1}{\lambda}\leq\frac{\int_{\partial B_r(x)}u^2}{\int_{\partial B_r(x)}v^2}\leq\lambda.\]
Assume by the contrary, there exist $x_i\in\{u=v\}$ and $r_i\geq1$
such that
\[\lim\limits_{i\to+\infty}\frac{\int_{\partial B_{r_i}(x_i)}u^2}{\int_{\partial B_{r_i}(x_i)}v^2}=0.\]
Define
\[u_i(x)=\frac{1}{L_i}u(x_i+r_ix),\ \ \ \ v_i(x)=\frac{1}{L_i}v(x_i+r_ix)\]
where $L_i$ is chosen so that
\[\int_{B_1(0)}u_i^2+v_i^2=1.\]
Then we can get a contradiction exactly as in the proof of the
previous lemma.
\end{proof}

\begin{coro}\label{coro 5.9}
For a solution $(u,v)$ of \eqref{equation} satisfying \eqref{linear
growth}, there exists a constant $C$ such that for every $x\in\{u=v\}$ and
$r\geq 1$,
\[\frac{1}{C}\leq J(r;x)\leq C.\]
\end{coro}
\begin{proof}
Assume $\sup_{B_r(0)}\left(u+v\right)\leq Cr$. Then for any fixed $x$, if $r$ is
large (depending on $|x|$),
\[\sup_{B_r(x)}\left(u+v\right)\leq 2Cr.\]
Combining this with \eqref{5.9}, we can obtain an upper bound of
$J(r;x)$. By the almost monotonicity of $J(r;x)$ (i.e. Corollary
\ref{coro 5.8}), this gives an upper bound of $J(r;x)$ for all
$x\in\{u=v\}$ and $r\geq 1$.
\par
Concerning the lower bound,  by the almost monotonicity of $J(r;x)$,
we only need to consider $J(1;x)$. Assume by the contrary, there
exist $x_i\in\{u=v\}$ such that
$\lim\limits_{i\to+\infty}J(1;x_i)=0$. Define $(u_i,v_i)$ as in the
proof of Lemma \ref{lem 5.7}. In particular, $(u_i,v_i)$ satisfies
the normalized condition \eqref{5.01}. Then
\[J(1;x_i)=L_i^4\int_{B_1(0)}\frac{|\nabla
u_i|^2+L_i^2u_i^2v_i^2}{|y|^{n-2}}dy \int_{B_1(0)}\frac{|\nabla
v_i|^2+L_i^2u_i^2v_i^2}{|y|^{n-2}}dy.\] Note that $L_i^2$ is the
parameter in the equations of $u_i$ and $v_i$, and by Lemma \ref{lem 5.7} it has a uniform positive lower bound.
\par
If $L_i$ are bounded, after passing to a subsequence, we can assume
$\lim\limits_{i\to+\infty}L_i=L_\infty>0$. Then by the proof of Lemma
\ref{lem 5.7}, $(u_i,v_i)$ converges to $(u_\infty,v_\infty)$
uniformly on any compact set of $\mathbb{R}^n$.
$(u_\infty,v_\infty)$ satisfies \eqref{equation scaled} with
$\kappa=L_\infty^2$. By passing to the limit in $J(1;x_i)$ we get
\[\int_{B_1(0)}\frac{|\nabla
u_\infty|^2+L_\infty^2u_\infty^2v_\infty^2}{|y|^{n-2}}dy\int_{B_1(0)}\frac{|\nabla
v_\infty|^2+L_\infty^2u_\infty^2v_\infty^2}{|y|^{n-2}}dy=0.\]
Without loss of generality we can assume the first integral is $0$.
 Hence
$u_\infty$ is a constant function. Moreover, if $v_\infty\neq 0$,
$u_\infty\equiv 0$. Using the equations of $u_\infty$ and $v_\infty$
and noting that $u_\infty(0)=v_\infty(0)$, we see both cases imply
$u_\infty\equiv v_\infty\equiv0$. This contradicts the normalization
condition \eqref{5.01}.
\par
If $L_i$ are unbounded, by Theorem \ref{lem 2.3}, there exists a
unit vector $e$ such that $(u_i,v_i)$ converges to
$(u_\infty,v_\infty)=((e\cdot x)^+,(e\cdot x)^-)$ uniformly on any
compact set of $\mathbb{R}^n$.

In this case, we still have the convergence
\begin{equation}\label{4.1}
\lim\limits_{\kappa\to+\infty}\int_{B_1(0)}\frac{|\nabla
u_i|^2+L_i^2u_i^2v_i^2}{|y|^{n-2}}dy=\int_{B_1(0)}\frac{|\nabla
u_\infty|^2}{|y|^{n-2}}dy.
\end{equation}
This is because, for every $\varepsilon>0$, by the strong
convergence of $u_i$ in $H^1_{\text{loc}}$ (similar to the proof of Proposition \ref{prop 4.3}),
\[\lim\limits_{\kappa\to+\infty}\int_{B_1(0)\setminus B_\varepsilon(0)}\frac{|\nabla
u_i|^2+L_i^2u_i^2v_i^2}{|y|^{n-2}}dy=\int_{B_1(0)\setminus
B_\varepsilon(0)}\frac{|\nabla u_\infty|^2}{|y|^{n-2}}dy,\] and by
\eqref{5.9} and the uniform H\"{o}lder continuity of $u_i$, as
$i\to+\infty$ and $\varepsilon\to 0$,
\[\int_{B_\varepsilon(0)}\frac{|\nabla
u_i|^2+L_i^2u_i^2v_i^2}{|y|^{n-2}}dy\leq C\varepsilon^{-n}\int_{
B_{2\varepsilon}(0)}u_i^2\leq C(u_i(0)^2+\varepsilon^{1/2})\to 0.\]

\eqref{4.1} allows us to pass to the limit in $J(1;x_i)$ to get
\[\int_{B_1(0)}\frac{|\nabla
u_\infty|^2}{|y|^{n-2}}dy\int_{B_1(0)}\frac{|\nabla
v_\infty|^2}{|y|^{n-2}}dy=0.\] Similar to the first case we get
$u_\infty\equiv v_\infty\equiv0$, a contradiction once again. So
there must exist a constant $C$ such that for all $x\in\{u=v\}$,
\[J(1;x)\geq C.\qedhere\]
\end{proof}
Note that in the above proof, if $L_i\to+\infty$, the limit is
$(u_\infty,v_\infty)=((e\cdot x)^+,(e\cdot x)^-)$ with a unit vector
$e$. In this case we have
\[J(1;u_\infty,v_\infty)=\lim\limits_{i\to+\infty}J(1;u_i,v_i)>0,\]
while by the upper bound on $J(r;x_i)$ we have
\[J(1;x_i)=L_i^4J(1;u_i,v_i)\leq C.\]
This is a contradiction. So $L_i$ must be uniformly bounded.
Combining this with Proposition \ref{doubling property}, we get
\begin{coro}\label{coro 5.10}
For a solution $(u,v)$ of \eqref{equation} satisfying \eqref{linear
growth}, there exists a constant $C$ such that for every $x\in\{u=v\}$,
\[\int_{\partial B_1(x)}u^2+v^2\leq C.\]
Consequently, for every $R>0$
\[\sup_{B_R(x)}\left(u+v\right)\leq C(1+R).\]
\end{coro}
This result can be viewed as the converse of Lemma \ref{lem 5.7}.

 \section{The global Lipschitz bound}
\numberwithin{equation}{section}
 \setcounter{equation}{0}
In this section we prove
\begin{thm}\label{thm 6.1}
Let $(u,v)$ be a solution of \eqref{equation} satisfying the linear
growth condition \eqref{linear growth}, then there is a constant
$C>0$ such that
\[\sup_{\mathbb{R}^n}\left(|\nabla u|+|\nabla v|\right)\leq C.\]
\end{thm}
The proof uses three lemmas. First we show that $u-v$ can be well
approximated by a linear function with unit slope in $B_R(x)$, if
$x\in\{u=v\}$ and $R$ is large enough (uniformly with respect to
$x\in\{u=v\}$).
\begin{lem}\label{lem 6.2}
For every $h\in(0,1/10)$, there exists an $R_0>0$ such that, for every
$x_0\in\{u=v\}$ and $R\geq R_0$, there exists a vector $e$ and a
constant $C$ ($C$ independent of $h$),
\[\frac{1}{C}\leq|e|\leq C,\]
such that
\begin{equation}\label{well approximation}
\sup_{B_R(x_0)}\left(|u-\left(e\cdot (x-x_0)\right)^+|+|v-\left(e\cdot
(x-x_0)\right)^-|\right)\leq hR.
\end{equation}
\end{lem}
\begin{proof}
Consider
\[u_R(x)=L(R)^{-1}u(x_0+Rx),\ \ \ \ v_R(x)=L(R)^{-1}v(x_0+Rx),\]
where $L(R)$ is chosen so that
\[\int_{\partial B_1(0)}u_R^2+v_R^2=1.\]
By the mean value inequality for subharmonic functions and Lemma
\ref{lem 5.7},
\[L(R)^2=\frac{1}{|\partial B_R|}\int_{\partial B_R(x_0)}u^2+v^2\geq C.\]
Since $(u_R,v_R)$ satisfies \eqref{equation scaled} with parameter
$L(R)^2R^2$, if $R\geq R_0=\sqrt{K}$, $(u_R,v_R)$ satisfies the
assumptions of Theorem \ref{lem 2.3}. Then we get a unit vector $e$
such that
\[\sup_{B_1(0)}\left(|u_R-(e\cdot x)^+|+|v_R-(e\cdot x)^-|\right)\leq h.\]
We claim that there exists a constant $C$ independent of $h$, such
that
\begin{eqnarray*}
&&J(1;0,u_R,v_R)\\
&=&\int_{B_1(0)}\frac{|\nabla u_R(y)|^2+L(R)^2R^2
 u_R(y)^2v_R(y)^2}{|y|^{n-2}}dy
\int_{B_1(0)}\frac{|\nabla v_R(y)|^2+L(R)^2R^2
u_R(y)^2v_R(y)^2}{|y|^{n-2}}dy\\
&\in&(\frac{1}{C},C).
\end{eqnarray*}
 The upper bound
can be obtained by using \eqref{5.9} and the uniform upper bound of
$u_R$ and $v_R$ in $B_2(0)$. The lower bound can be got by
restricting the first integral to the domain $\{e\cdot x\geq
C(n)h\}\cap B_1(0)$, where $C(n)$ is chosen large enough so that in
$\{e\cdot x\geq C(n)h\}\cap B_1(0)$, $|\nabla u_R-e|\leq 1/4$ and
$L(R)^2R^2
 u_R(y)^2v_R(y)^2\leq e^{-R}$.
(This is an application of Lemma \ref{estimate exponential decay}.)
Similar estimates hold for the second integral.
\par
On the other hand, by rescaling the estimate in Corollary \ref{coro
5.9}, we get
\[\frac{1}{C}\frac{R^2}{L(R)^2}\leq J(1;0,u_R,v_R)=\frac{R^2}{L(R)^2}J(R;x_0,u,v)\leq C\frac{R^2}{L(R)^2}.\]
Combining these two we see
\[\frac{1}{C}\leq\frac{L(R)}{R}\leq C.\]
A suitable rescaling gives the required claim.
\end{proof}

The next estimate is a standard interior gradient estimate for
elliptic equations (cf. \cite{G-T}).
\begin{lem}\label{large R}
For every $ K>0$ and $R>1$, there exists a constant $C(K,R)>0$, if
$(u_\kappa,v_\kappa)$ satisfies \eqref{equation scaled} in
$B_{R}(0)$ with $\kappa\leq K$, and
$$\sup\limits_{B_R(0)}\left(|u_\kappa-x_1^+|+|v_\kappa-x_1^-|\right)\leq 1.$$
Then
$$\sup\limits_{B_{\frac{R}{2}}(0)}\left(|\nabla u_\kappa|+|\nabla v_\kappa|\right)\leq C(K,R).$$
\end{lem}

Finally, for solutions appearing in Lemma \ref{well approximation},
in the good part (far away from $\{e\cdot x=0\}$), we have the
following bound on the gradient.
\begin{lem}\label{large k}
For every $h_0\in(0,1/100)$, there exists a $K_2>0$, if
$(u_\kappa,v_\kappa)$ satisfies \eqref{equation scaled}
in $B_1(0)$ with $\kappa\geq K_2$, 
and
$$\sup\limits_{B_1(0)}\left(|u_\kappa-x_1^+|+|v_\kappa-x_1^-|\right)\leq h_0,$$
then there is a constant $C(n)$, which depends on $n$ only, such
that
$$\sup\limits_{B_{1/2}(0)\setminus\{|x_1|\geq 4h_0\}}\left(|\nabla u_\kappa|+|\nabla v_\kappa|\right)\leq C(n).$$
\end{lem}
\begin{proof}
Take an $x_0\in B_{1/2}(0)\setminus\{x_1\geq 4h_0\}$. In
$B_{2h_0}(x_0)$,
\[u_\kappa\geq h_0\geq v_\kappa.\]
By Lemma \ref{estimate exponential decay}, we have
\[\sup_{B_{h_0}(x_0)}v_\kappa\leq C_1(n)h_0e^{-C_2(n)\kappa^{1/2} h_0^{3/2}}.\]
If $K_2$ is large enough and $\kappa\geq K_2$, there exists a
constant $C(n)$ such that
\[\sup_{B_{h_0}(x_0)}\left(\Delta u_\kappa+\Delta v_\kappa\right)\leq C(n).\]
Note that we always have $\Delta u_\kappa\geq0$, $\Delta v_\kappa\geq 0$. Then
\[|\nabla u_\kappa(x_0)|\leq C\left[\frac{1}{h_0}
\left(\sup_{B_{h_0}(x_0)}u_\kappa-\inf_{B_{h_0}(x_0)}u_\kappa\right)+h_0\sup_{B_{h_0}(x_0) }\Delta
u_\kappa\right]\leq C.\]
Similar estimates hold for $v_\kappa$.
\end{proof}
\begin{proof}[Proof of Theorem \ref{thm 6.1}]
For every $y_0\in \mathbb{R}^n$. Take $R/2=\text{dist}(y_0,\{u=v\})$
and an $x_0\in \{u=v\}$ realizing this distance. Without loss of
generality, assume $u>v$ in $B_{R/2}(y_0)$.

If $R\leq R_0$, Corollary \ref{coro 5.10} implies
\[\sup_{B_{R/2}(y_0)}\left(u+v\right)\leq C(R_0).\]
Then as in Lemma \ref{large R}, we can get an upper bound of
$|\nabla u(y_0)|+|\nabla v(y_0)|$, which depends only on $R_0$ and
some universal constants.

If $R\geq R_0$, Lemma \ref{lem 6.2} is applicable. So there exist a
vector $e$ and a small $h$ determined by $R_0$, such that
\eqref{well approximation} holds. Because $u>v$ in $B_{R/2}(y_0)$,
$\{e\cdot (x-x_0)<-2hR\}\cap B_{R/2}(y_0)=\emptyset$. Assuming
$h<1/10$, then we have $u-v\geq R/8$ in $B_{R/4}(y_0)$.
 Define
\[u_R(x)=R^{-1}u(x_0+Rx),\ \ \ \ v_R(x)=R^{-1}v(x_0+Rx).\]
Denote $z_0=R^{-1}(y_0-x_0)\in B_{1/2}(0)$. We can apply Lemma
\ref{large k} to $(u_R,v_R)$ to get a uniform upper bound of
$|\nabla u_R(z_0)|+|\nabla v_R(z_0)|$. This implies an upper bound
of $|\nabla u(y_0)|+|\nabla v(y_0)|$, which depends only on $h$ and
some universal constants.
\end{proof}

\section{Comparison with harmonic functions}
\numberwithin{equation}{section}
\setcounter{equation}{0}

In this section we first present some consequences of the global
Lipschitz continuity.  Note that it is possible to prove these
results directly as in the proof of Theorem \ref{thm 6.1}. However, for
simplicity of presentation, we state these results as corollaries
of the global Lipschitz continuity. These results hold for all solutions $(u,v)$ of \eqref{equation} satisfying \eqref{linear
growth}. Then by further assuming that $(u,v)$ is a minimizer (in the sense of \eqref{local energy minimizer}), we derive an energy estimate by comparing $u-v$ with a harmonic replacement. This estimate will play a crucial role in the next step of the proof of our main result.

\begin{lem}\label{lem 7.1}
There exists a universal constant $C$ such that
\[uv\leq C~~\text{in}~~\mathbb{R}^n.\]
\end{lem}
\begin{proof}
Fix a large constant $M>0$. Take an $x_0\in\mathbb{R}^n$. Without
loss of generality we can assume $u(x_0)=A>M$. Because $|\nabla
u|+|\nabla v|\leq C$,
\[u\geq A-C>0~~\text{in}~~B_1(x_0),\]
and
\[\Delta v\geq (A-C)^2v~~\text{in}~~B_1(x_0).\]
By Lemma \ref{estimate exponential decay} and the Lipschitz
continuity of $v$, we have
\[v(x_0)\leq C_1(v(x_0)+C)e^{-C_2\sqrt{A-C}}.\]
This implies
\[v(x_0)\leq Ce^{-\frac{A}{C}}.\]
Hence
\[u(x_0)v(x_0)\leq CAe^{-\frac{A}{C}}\leq C(M),\]
for a  constant $C(M)$ depending only on $M$.
\end{proof}
The same method gives the following
\begin{coro}\label{coro 7.1}
There exists a universal constant $C$ such that
\[uv^2+u^2v\leq C~~\text{in}~~\mathbb{R}^n.\]
\end{coro}
\begin{lem}\label{lem 7.2}
\[u|\nabla v|+v|\nabla u|\leq C~~\text{in}~~\mathbb{R}^n.\]
\end{lem}
\begin{proof}
Fix a large constant $M>0$. Take an $x_0\in\mathbb{R}^n$. Because
$|\nabla v|\leq C$, we can assume $u(x_0)=A>M$. The proof of the
previous lemma in fact shows
\[v\leq Ce^{-\frac{A}{C}}~~\text{in}~~B_1(x_0).\]
Then by the gradient estimate of elliptic equations and the equation
of $v$ we get
\[|\nabla v(x_0)|\leq\sup_{B_1(x_0)}(v+\Delta v)\leq Ce^{-\frac{A}{C}}.\]
So
\[u(x_0)|\nabla v(x_0)|\leq CAe^{-\frac{A}{C}}\leq C.\qedhere\]
\end{proof}
\begin{lem}
For every ball $B_R(x_0)$,
\[\int_{B_R(x_0)}u^2v^2\leq CR^{n-1}.\]
\end{lem}
\begin{proof}
Note that
$$\int_{B_R(x_0)}uv^2=\int_{B_R(x_0)}\Delta u=\int_{\partial B_R(x_0)}u_r
\leq \int_{\partial B_R(x_0)}|\nabla u|\leq CR^{n-1}.$$ The same
estimate holds for $vu^2$.
\par
Now by  Lemma \ref{lem 7.1}, for every $x$ either $u(x)$ or
$v(x)\leq C$, so \[\int_{B_R(x_0)}u^2v^2\leq
C\int_{B_R(x_0)}u^2v+uv^2\leq CR^{n-1}.\qedhere\]
\end{proof}
For every ball $B_R(x_0)$, let $\varphi_{R,x_0}$ be the solution of
the problem
\begin{equation*}
\left\{ \begin{aligned}
 &\Delta \varphi_{R,x_0}=0~~\text{in}~~B_R(x_0),\\
 &\varphi_{R,x_0}=u-v~~\text{on}~~\partial B_R(x_0).
                          \end{aligned} \right.
\end{equation*}
Because $u$ and $v$ are smooth in $\mathbb{R}^n$,
$\varphi_{R,x_0}\in C^{\infty}(\overline{B_R(x_0)})$. It may be true
that $\varphi_{R,x_0}$ are uniformly Lipschitz, that is, there is a
universal constant $C>0$ such that
\begin{equation*}
\sup_{B_{R,x_0}}|\nabla\varphi_R(x_0)|\leq C.
\end{equation*}
However currently we do not know how to prove this. Instead we give
a weaker result, which is sufficient for our use.
\begin{lem}\label{lem 6.5}
For every $\delta\in(0,1)$, there exists a universal constant $C$
such that for every $R\geq 1$ and $x_0\in\mathbb{R}^n$,
\begin{equation}\label{Lip for harmonic extenstion}
\sup_{B_R(x_0)}|\nabla\varphi_{R,x_0}|\leq CR^{\delta}.
\end{equation}
\end{lem}
\begin{proof}
Because $u$ and $v$ are globally Lipschitz continuous,
\[\tilde{u}(x):=R^{-1}\left(u(x_0+Rx)-u(x_0)\right),\ \ \ \ \tilde{v}(x):=R^{-1}\left
(v(x_0+Rx)-v(x_0)\right)\] are uniformly bounded in
$Lip(\overline{B_1(0)})$. By the global H\"{o}lder continuity
estimate applied to harmonic functions (see \cite{G-T}),
\[\tilde{\varphi}(x):=\frac{1}{R}\left(\varphi_{R,x_0}(x_0+Rx)-u(x_0)+v(x_0)\right)\]
is uniformly bounded in $C^{1-\delta}(\overline{B_1(0)})$.

By noting the boundary condition of $\tilde{\varphi}$, we get a
universal constant $C$ such that
\[|\tilde{u}(x)-\tilde{v}(x)-\tilde{\varphi}(x)|\leq
C(1-|x|)^{1-\delta}~~\text{in}~~B_1(0).\]
Rescaling back we get
\begin{equation}\label{7.1}
|u(x)-v(x)-\varphi_{R,x_0}(x)|\leq CR^\delta(R-|x-x_0|)^{1-\delta}~~\text{in}~~B_R(x_0).
\end{equation}
By the boundary gradient estimate \cite{G-T}, for every
$x\in\partial B_R(x_0)$, we have
\begin{eqnarray*}
|\nabla(u-v-\varphi_{R,x_0})(x)| &\leq& C\sup_{B_1(x)\cap
B_R(x_0)}|\Delta(u-v-\varphi_{R,x_0})|+C\sup_{B_1(x)\cap
B_R(x_0)}|u-v-\varphi_{R,x_0}|\\
&\leq& CR^\delta.
\end{eqnarray*}
In the above we have used Corollary \ref{coro 7.1} to estimate
$\Delta(u-v)$.

Since $|\nabla u|+|\nabla v|\leq C$ for a universal constant, by
choosing a larger constant, we get
\[\sup_{\partial B_R(x_0)}|\nabla \varphi_{R,x_0}|\leq CR^\delta+C\leq CR^{\delta}.\]
\eqref{Lip for harmonic extenstion} follows by applying the maximum
principle to $|\nabla\varphi_{R,x_0}|$.
\end{proof}

\begin{prop}\label{lem comparison of the energy}
There exists a constant $C$, such that for all $x_0\in\mathbb{R}^n$ and $R\geq 1$,
\[\int_{B_R(x_0)}|\nabla(u-v-\varphi_{R,x_0})|^2\leq CR^{n-\frac{1}{2}}.\]
\end{prop}
\begin{proof}
Fix a $\delta\in(0,1/8)$ and take a constant $C$ so that the
previous lemma holds. For simplicity, we assume $x_0=0$ and denote
$\varphi_{R,x_0}$ by $\varphi$. Direct calculations give
\[\int_{B_R}|\nabla(u-v-\varphi)|^2=\int_{B_R}|\nabla u|^2+|\nabla v|^2+|\nabla\varphi|^2
-2\nabla u\nabla v-2\nabla u\nabla\varphi+2\nabla v\nabla\varphi.\]
We divide the estimate into three parts.

{\bf Step 1.} An integration by parts gives
\[\int_{B_R}\nabla (u-v)\nabla\varphi=\int_{\partial
B_R}\varphi\frac{\partial\varphi}{\partial r}\\
=\int_{B_R}|\nabla\varphi|^2.
\]

{\bf Step 2.} Take $\bar{u}=\varphi^+$, $\bar{v}=\varphi^-$ in
$B_{R-1}$, $\bar{u}=(R-|x|)\varphi^++(|x|-R+1)u(x)$,
$\bar{v}=(R-|x|)\varphi^-+(|x|-R+1)v(x)$ in $B_R\setminus B_{R-1}$.
Noting that $\bar{u}=u$, $\bar{v}=v$ on $\partial B_R$, hence by the
locally energy minimizing property of $(u,v)$ we get
\begin{eqnarray*}
\int_{B_R}|\nabla u|^2+|\nabla v|^2+u^2v^2 &\leq& \int_{B_R}|\nabla
\bar{u}|^2+|\nabla
\bar{v}|^2+\bar{u}^2\bar{v}^2\\
&=&\int_{B_{R-1}}|\nabla\varphi|^2+\int_{B_R\setminus
B_{R-1}}|\nabla \bar{u}|^2+|\nabla \bar{v}|^2+\bar{u}^2\bar{v}^2.
\end{eqnarray*}
We need to estimate the last integral.

In $B_R\setminus B_{R-1}$, similar to the derivation of \eqref{7.1},
we have
\begin{equation}\label{7.2}
|u-\varphi^+|+|v-\varphi^-|\leq CR^\delta.
\end{equation}
 Combining this with
\eqref{Lip for harmonic extenstion} we get, for $x\in B_R\setminus
B_{R-1}$,
\begin{equation}\label{7.3}
|\nabla\bar{u}(x)|\leq (R-|x|)|\nabla\varphi^+(x)|+(|x|-R+1)|\nabla
u(x)| +\big|\nabla |x|\big||u(x)-\varphi^+(x)|\leq CR^\delta.
\end{equation}
Similarly $|\nabla\bar{v}|\leq CR^\delta$ in  $B_R\setminus
B_{R-1}$.

We claim that in $B_R\setminus B_{R-1}$,
 \begin{equation}\label{7.4}
\bar{u}\bar{v}\leq(\varphi^++u)(\varphi^-+v)\leq\varphi^+v+\varphi^-u+uv\leq
CR^{2\delta}.
\end{equation}
In view of Lemma \ref{lem 7.1}, we only need to estimate
$\varphi^+v$. There are two cases. If
$\varphi^+(x)\in(0,2CR^\delta)$ ($C$ as in \eqref{7.2}), then
because $\varphi^-(x)=0$, \eqref{7.2} implies $v(x)\leq CR^\delta$.
So
\[\varphi^+(x)v(x)\leq 2C^2R^{2\delta}.\]
If $\varphi^+(x)\geq 2CR^{\delta}$, then again by \eqref{7.2},
\[u(x)\geq \varphi^+(x)-CR^\delta\geq\frac{1}{2}\varphi^+(x).\]
So by Lemma \ref{lem 7.1},
\[v(x)\leq\frac{C}{u(x)}\leq\frac{2C}{\varphi^+(x)},\]
which again implies \eqref{7.4}. This finishes the proof of the
claim.

 Combining \eqref{7.3} and \eqref{7.4} we get
\[\int_{B_R\setminus
B_{R-1}}|\nabla \bar{u}|^2+|\nabla \bar{v}|^2+\bar{u}^2\bar{v}^2\leq
CR^{n-1+4\delta}\leq CR^{n-\frac{1}{2}}.\]
 This implies
\[\int_{B_R}|\nabla u|^2+|\nabla v|^2+u^2v^2\leq\int_{B_R}|\nabla
\varphi|^2+CR^{n-\frac{1}{2}}.\]

{\bf Step 3.} Finally, let us estimate
\begin{eqnarray*}
\big|\int_{B_R}\nabla u\nabla v\big|&=&\big|\int_{\partial
B_R}\frac{\partial
u}{\partial r}v -\int_{B_R}v\Delta u \big|\\
&\leq& \int_{\partial B_R}|\nabla u|v +\int_{B_R} v\Delta u\\
&\leq& \int_{B_R} v\Delta u+CR^{n-1}.
\end{eqnarray*}
Here we have used Lemma \ref{lem 7.2} to estimate the boundary
integral.

Denote the measure $\mu=\Delta u dx$. Note that for every $t>0$,
\[\mu(\{v>t\}\cap B_R)\leq\mu(B_R)\leq CR^{n-1}\]
and because in $\{v>t\}$, $u\leq Ce^{-\frac{t}{C}}$ and $\Delta
u\leq Ce^{-\frac{t}{C}}$ for a universal constant $C$ (by Lemma
\ref{estimate exponential decay}),
\[\mu(\{v>t\}\cap B_R)\leq Ce^{-\frac{t}{C}}R^n.\]
Now if $R$ is large enough, by choosing $M=C\log R$ and dividing the
integration into two parts, we get
\begin{eqnarray*}
\int_{B_R}vd\mu&=&\int_0^{+\infty}\mu(\{v>t\}\cap B_R)dt\\
&\leq&\int_0^MCR^{n-1}dt+\int_M^{+\infty}Ce^{-\frac{t}{C}}R^ndt\\
&\leq& CR^{n-1}(\log R+1)\\
&\leq & CR^{n-1/2}.
\end{eqnarray*}

 Putting Step 1 to 3 together we finish the
proof.
\end{proof}

\section{Uniqueness of the asymptotic cone at infinity}
\numberwithin{equation}{section}
\setcounter{equation}{0}

In this section $(u,v)$ is a solution of \eqref{equation} on
$\mathbb{R}^n$, satisfying the linear growth condition \eqref{linear growth} and the minimizing condition \eqref{local energy minimizer}. We will use Proposition \ref{lem comparison of the energy} to refine the results from Section 3.

Fix a base point $x_0$. By Proposition \ref{lem comparison of the
energy}, for every $R\geq 1$,
\[\int_{B_R(x_0)}|\nabla\varphi_{R,x_0}-\nabla\varphi_{2R,x_0}|^2\leq CR^{n-\frac{1}{2}}.\]
Since both $\nabla\varphi_{R,x_0}$ and $\nabla\varphi_{R,x_0}$ are
harmonic functions, we get
\begin{equation}\label{6.1}
\sup\limits_{B_{R/2}(x_0)}|\nabla\varphi_{R,x_0}-\nabla\varphi_{2R,x_0}|\leq
CR^{-\frac{1}{4}}.
\end{equation}
Hence for every fixed $r$ and $i$ large,
\[\sup\limits_{B_r(x_0)}|\nabla\varphi_{2^i,x_0}-\nabla\varphi_{2^{i+1},x_0}|\leq
C2^{-\frac{i}{4}}.\] Adding in $i$ we see
$\lim\limits_{i\to+\infty}\nabla\varphi_{2^i,x_0}=\nabla\varphi_\infty$
exists. By Lemma \ref{lem 6.5}, for any $R\geq 1$,
\[\sup_{B_R(0)}|\nabla\varphi_\infty|\leq CR^{\frac{1}{2}}+C\sum_{i\geq 1}2^{-\frac{i}{4}}
\leq CR^{\frac{1}{2}}.\] Since $\nabla\varphi_\infty$ is an entire
harmonic vector-valued function, by the Liouville theorem, it is a
constant vector function. In conclusion, there exists a constant
vector $e(x_0)$ such that
\[\lim\limits_{i\to+\infty}\nabla\varphi_{2^i,x_0}=e(x_0)\ \ \mbox{uniformly on any compact set of}\ \mathbb{R}^n.\]

Furthermore, by choosing $i_0$ such that $2^{i_0-1}<R\leq 2^i_0$,
adding in $i$ from $i_0$ to $+\infty$ and using \eqref{6.1}, we have
\[\sup\limits_{B_{R/2}(x_0)}|\nabla\varphi_{R,x_0}-e(x_0)|\leq
CR^{-\frac{1}{4}}.\] Substituting this into Proposition \ref{lem
comparison of the energy}, we have for every $R\geq 1$ and
$x_0\in\mathbb{R}^n$,
\begin{equation}
\int_{B_R(x_0)}|\nabla(u-v)-e(x_0)|^2\leq CR^{n-\frac{1}{2}}.
\end{equation}
Combining this fact with the result in Section 3 and Remark \ref{rmk 5.5}, we know
\[\lim\limits_{R\to+\infty}R^{-1}u(x_0+Rx)=(e(x_0)\cdot x)^+,\ \
\ \lim\limits_{R\to+\infty}R^{-1}v(x_0+Rx)=(e(x_0)\cdot x)^+,\]
uniformly on any compact set of $\mathbb{R}^n$.
\par
Next we note that such $e(x_0)$ is independent of the base point
$x_0$. This is because we also have
\[\lim\limits_{R\to+\infty}\frac{u(x_0+Rx)}{R}=
\lim\limits_{R\to+\infty}\frac{u(R(x+\frac{x_0}{R}))}{R}=(e(0)\cdot
x)^+\] uniformly on any compact set of $\mathbb{R}^n$.
So $e(x_0)=e(0)$, which is independent of $x_0$.
\par
In conclusion, we have proved
\begin{prop}\label{prop 6.1}
There exists a vector $e_0$, such that for every $R\geq 1$ and
$x_0\in\mathbb{R}^n$,
\begin{equation}\label{6.2}
\int_{B_R(x_0)}|\nabla(u-v)-e_0|^2\leq CR^{n-\frac{1}{2}}.
\end{equation}
\end{prop}
By Remark \ref{rmk 5.5}, we can take $e_0$ such that $|e_0|=1$. In
the following of this paper we will assume $e_0=e_n$, the $n$-th
coordinate direction.

\section{Existence of a cone of monotonicity}
\numberwithin{equation}{section}
 \setcounter{equation}{0}

In this section $(u,v)$ denotes a solution of \eqref{equation} on
$\mathbb{R}^n$, satisfying the linear growth condition \eqref{linear growth} and the minimizing condition \eqref{local energy minimizer}. We will combine the sliding method (in the spirit of \cite{F}) with results from previous sections to establish the existence of a cone of monotonicity for $u$ and $v$.

In the following we denote $x=(x^\prime,x_n)$ where
$x^\prime=(x_1,\cdots,x_{n-1})$.
\begin{lem}\label{lem 9.1}
There exists a $M>0$ such that if $|u(x_0)-v(x_0)|\geq M$, then
\[|\nabla(u-v)(x_0)-e_n|\leq\frac{1}{4}.\]
\end{lem}
\begin{proof}
Assume $u(x_0)-v(x_0)=A>M$. Denote
\[\widetilde{u}(x)=A^{-1}u(x_0+Ax),\ \ \ \ \widetilde{v}(x)=A^{-1}v(x_0+Ax).\]
Then by Lemma \ref{lem 7.1} and the Lipschitz continuity of $u$ and
$v$, in $B_{1/C}(0)$,
\[\Delta\widetilde{u}(x)=A^4\widetilde{u}(x)\widetilde{v}(x)^2\leq
Ce^{-\frac{\sqrt{A}}{C}},\ \ \ \
\Delta\widetilde{v}(x)=A^4\widetilde{u}(x)^2\widetilde{v}(x)\leq
Ce^{-\frac{\sqrt{A}}{C}}.\] By \eqref{6.2}, we also have
\[\int_{B_{1/C}(0)}|\nabla(\widetilde{u}-\widetilde{v})-e_n|^2\leq CA^{-\frac{1}{4}}.\]
 If $M$ is large
enough, the $W^{2,p}$ estimate (for $p>n$) and the Sobolev embedding
theorem imply
\[|\nabla(\widetilde{u}-\widetilde{v})(0)-e_n|\leq\frac{1}{4}.\]
Rescaling back to $u$ and $v$ we get the claim.
\end{proof}
The constant $1/4$ can be made arbitrarily small if we choose $M$
large enough.

\begin{coro}\label{coro 9.2}
If $u(x)-v(x)\geq M$, then $|\nabla u(x)-e_n|\leq 1/2$. If
$v(x)-u(x)\geq M$, then $|\nabla v(x)-e_n|\leq 1/2$.
\end{coro}
\begin{proof}
Assume $M$ is large enough. If $u(x)-v(x)\geq M$, then by Lemma
\ref{lem 7.2}, $|\nabla v(x)|\leq1/4$. Hence
 \[|\nabla u(x)-e_n|\leq|\nabla(u-v)(x)-e_n|+|\nabla v(x)|\leq 1/2.\qedhere\]
\end{proof}
\begin{coro}
For $|t|>M$, $\{u-v=t\}$ is a Lipschitz graph of the form
$\{x_n=f(x^\prime,t)\}$, with the Lipschitz constant of $f$ (with
respect to $x^\prime$) less than $4$.
\end{coro}
\begin{proof}
We need to note that for every
$x_0=(x_0^\prime,0)\in\mathbb{R}^{n-1}$,
\[\lim\limits_{R\to+\infty}R^{-1}\left(u(x_0+Rx)-v(x_0+Rx)\right)=x_n.\]
Hence
\[\lim\limits_{x_n\to+\infty}\left(u(x_0^\prime,x_n)-v(x_0^\prime,x_n)\right)=+\infty,\ \ \ \
\lim\limits_{x_n\to-\infty}\left(u(x_0^\prime,x_n)-v(x_0^\prime,x_n)\right)=-\infty.\]
Since
 $u-v$ is strictly increasing in the direction of $e_n$ in
 $\{|u-v|>M\}$,
for every $|t|>M$ and $x^\prime\in\mathbb{R}^{n-1}$, there is a
unique finite $f(x^\prime,t)$ such that
$(u-v)(x^\prime,f(x^\prime,t))=t$. \footnote{ For example, if
$u(x^\prime,x_n)-v(x^\prime,x_n)=t>M$, then
$\frac{\partial(u-v)}{\partial x_n}(x^\prime, x_n)>0$, so
$u(x^\prime,y_n)-v(x^\prime,y_n)>t$ for those $y_n>x_n$ close to
$x_n$. Then by continuation using the monotonicity of $u-v$ in
$\{u-v>M\}$, $u(x^\prime,y_n)-v(x^\prime,y_n)>t$ for all $y_n>x_n$.}

There is a cone of directions
$$C(e_n,1/4):=\{e\in\mathbb{R}^n,
|e|=1, e\cdot e_n\geq 1/4\},$$ such that for every $e\in
C(e_n,1/4)$, and $x\in\{|u-v|\geq M\}$,
\[\nabla(u-v)(x)\cdot e\geq 0.\]
That is, in $\{|u-v|\geq M\}$, $u-v$ is monotone increasing along
the directions in $C(e_n,1/4)$. This gives the Lipschitz continuity
of $f$.
\end{proof}
In the following we denote the level set $\{u-v=t\}$ by $\Gamma_t$.
By the previous lemma, for $|t|\geq M$,
$\Gamma_t=\{x_n=f(x^\prime,t)\}$. The following result states that
the width (in the $x_n$ direction) of $\{|u-v|<M\}$ is bounded.
\begin{lem}\label{lem 9.4}
There is a universal constant $C$, if $M$ is large enough, then
\[\sup_{x^\prime\in\mathbb{R}^{n-1}}\left(f(x^\prime,M)-f(x^\prime,-M)\right)\leq CM.\]
\end{lem}
\begin{proof}
Take an arbitrary $x_0=(x_0^\prime,f(x_0^\prime,M))\in\Gamma_M$. By
Lemma \ref{lem 6.2} and Proposition \ref{prop 6.1}, if $M$ is large
enough (but independent of $x_0$),
\[\sup_{B_3(0)}\big|\frac{1}{M}u(x_0+Mx)-\frac{1}{M}v(x_0+Mx)-x_n-1\big|\leq 1.\]
Hence
\[\inf_{B_3(0)}\left[\frac{1}{M}u(x_0+Mx)-\frac{1}{M}v(x_0+Mx)\right]\leq -3+1+1\leq -1.\]
That is, $B_{3M}(x_0)\cap\{u-v\leq-M\}\neq\emptyset$, or
$\text{dist}(x_0,\Gamma_{-M})\leq 3M$.
\par
Take $y_0=(y^\prime_0,f(y^\prime_0,-M))\in
B_{3M}(x_0)\cap\Gamma_{-M}$. By the monotonicity of $u-v$ in the
direction of $C(e_n,1/4)$ in $\{u-v\leq-M\}$, for a sufficiently
large universal constant $C$,
\[\{(x^\prime,x_n):x_n<
f(y^\prime_0,-M)-C|x^\prime-y^\prime_0|\}\subset\{u-v<-M\}.\] Thus
for $(x^\prime_0, x_n)$ such that
\[x_n\leq f(y^\prime_0,-M)-C|x^\prime_0-y^\prime_0|,\]
we have $(x^\prime_0, x_n)\in\{u-v<-M\}$. In other words,
\[f(x^\prime_0,-M)\geq f(y^\prime_0,-M)-C|x^\prime_0-y^\prime_0|.\]
Note that
\[|f(y^\prime_0,-M)-f(x_0^\prime,M)|+|x^\prime_0-y^\prime_0|\leq 3M.\]
Combining these two inequalities we get
\[f(x^\prime_0,-M)\geq f(x_0^\prime,M)-CM.\qedhere\]
\end{proof}

\begin{prop}\label{lem 9.5}
For each $\tau\in C(e_n,3/4)$, we have
\[u_\tau=\tau\cdot\nabla u\geq 0;\ \ \ \ v_\tau=\tau\cdot\nabla v\leq 0~~\text{in}~~\mathbb{R}^n.\]
\end{prop}
\begin{proof}
For $t\geq 0$, define
\[u^t(x):=u(x+t\tau),\ \ \ \  v^t(x):=v(x+t\tau).\]
We will prove that for all $t\geq 0$,
\begin{equation}\label{9.1}
u^t(x)\geq u(x),\ \ \ \ v^t(x)\leq v(x)~~\text{in}~~\mathbb{R}^n.
\end{equation}

By Corollary \ref{coro 9.2}, $u_\tau\geq 0$ in $\{u-v>M\}$ and
$v_\tau\leq 0$ in $\{u-v\leq -M\}$. So for every $t\geq 0$, $u^t\geq
u$ in $\{u-v>M\}$ and $v^t\leq v$ in $\{u-v\leq -M\}$.
\par
{\bf Step 1.} \eqref{9.1} holds for $t$ large enough.\\
By the previous lemma, for $x\in\{-M\leq u-v\leq M\}$, if $t\geq CM$, $x+t e_n\in\{u-v\geq
2M\}$. Because $\tau\in C(e_n,3/4)\subset\subset C(e_n,1/4)$, if
$t\geq (C+64)M$, $x+t\tau\in x+CMe_n+C(e_n,1/4)\subset\{u-v\geq
2M\}$. Since $-M\leq u(x)-v(x)\leq M$,
by Lemma \ref{lem 7.1} we have
\[u(x)\leq M+v(x)\leq M+\frac{C}{u(x)}\leq2M\leq u(x+t\tau).\]
That is, $u^t\geq u$ in $\{-M\leq u-v\leq M\}$.
\par
Next we show that $u^t\geq u$ in $\Omega=\{u-v<-M\}$. By Lemma
\ref{lem 7.1} and the global Lipschitz continuity, $u$ and $u^t$ are
bounded in $\Omega$. Assume $\inf_\Omega(u^t-u)<0$. Take a sequence
of $x_k\in\Omega$ such that
$$\lim\limits_{k\to+\infty}\left(u^t(x_k)-u(x_k)\right)=\inf_\Omega\left(u^t-u\right)<0.$$
Because
\[\liminf\limits_{k\to+\infty}u(x_k)\geq
\lim\limits_{k\to+\infty}\left(u(x_k)-u^t(x_k)\right)>0,\] there exists a
$\delta>0$ such that $u(x_k)\geq\delta$ for all $k$. By Lemma
\ref{lem 7.1}, $v(x_k)\leq C/\delta$. Now define
\[u_k(x)=u(x_k+x),\ \ \ \ v_k(x)=v(x_k+x).\]
Since $u_k(0)$ and $v_k(0)$ are uniformly bounded and
$\sup\limits_{\mathbb{R}^n}\left(|\nabla u_k|+|\nabla v_k|\right)\leq C$, there
exist two continuous functions $u_\infty$ and $v_\infty$ such that
\[u_k\to u_\infty, \ \ \ v_k\to v_\infty~~\text{uniformly on any compact set of}~~\mathbb{R}^n.\]
$(u_\infty, v_\infty)$ is a solution of \eqref{equation}. They also
satisfy $u_\infty^t(x)=u_\infty(x+t\tau)\geq u_\infty(x)$ in
$\{u_\infty-v_\infty\geq -M\}$ and $v^t_\infty(x)\leq v_\infty(x)$
in $\{u_\infty-v_\infty\leq -M\}$. Then
\begin{equation*}
\left\{ \begin{aligned}
 &\Delta u_\infty=u_\infty v_\infty^2~~\text{in}~~\{u_\infty-v_\infty<-M\}, \\
 &\Delta u_\infty^t=u_\infty^t(v_\infty^t)^2\leq u_\infty^tv_\infty^2~~\text{in}~~\{u_\infty-v_\infty<-M\}, \\
 & u_\infty^t\geq u_\infty~~\text{on}~~\partial\{u_\infty-v_\infty< -M\}.
                          \end{aligned} \right.
\end{equation*}
Moreover,
\[u_\infty^t(0)-u_\infty(0)=\inf\limits_{\{u_\infty-v_\infty< -M\}}\left(u_\infty^t-u_\infty\right)<0.\]
A direct application of the maximum principle gives a
contradiction.
\par
We can use similar method to show that $v^t\leq v$ in $\{u-v\geq
-M\}$.
\par
{\bf Step 2.} By the result in Step 1, we can define
 \[t_0:=\inf\{t: \eqref{9.1}~~\text{holds for all}~~s>t\}.\]
Assume $t_0>0$. Since $u^{t_0}\geq u$ and $v^{t_0}\leq v$ in
$\mathbb{R}^n$, we have
\begin{equation*}
\left\{ \begin{aligned}
 &\Delta u^{t_0}=u^{t_0} (v^{t_0})^2\leq u^{t_0}v^2, \\
 &\Delta v^{t_0}=v^{t_0} (u^{t_0})^2\geq v^{t_0}u^2.
                          \end{aligned} \right.
\end{equation*}
Comparing with $u$ and $v$, the strong maximum principle implies
that
\[u^{t_0}>u, \ \ \ v^{t_0}<v~~\text{strictly in}~~\mathbb{R}^n.\]
Here we need to note that, there is a constant $\delta>0$ such that
in $\{u-v>M\}$, by Corollary \ref{coro 9.2}
\[u^{t_0}(x)-u(x)=\int_0^{t_0}\tau\cdot\nabla u(x+t\tau)dt\geq\delta.\]
And similarly in $\{u-v<-M\}$,
\[v^{t_0}(x)-v(x)\leq-\delta.\]
These two estimates imply that
\[u^{t_0}-v^{t_0}\geq u-v+\delta~~\text{in}~~\{|u-v|>M\}.\]
Then similar to the method in Step 1, we can prove
\[\inf\limits_{\mathbb{R}^n}\left[\left(u^{t_0}-v^{t_0}\right)-\left(u-v\right)\right]>0.\]
By the global Lipschitz continuity of $u-v$, there exists a
$\varepsilon>0$, such that for all $t\in(t_0-\varepsilon,t_0)$,
\begin{equation}\label{8.2}
\inf\limits_{\mathbb{R}^n}\left[\left(u^t-v^t\right)-\left(u-v\right)\right]>0.
\end{equation}
We claim that for such $t$, \eqref{9.1} still holds. We only prove
the inequality for $u$ and the other one is similar. As before we
only need to consider the set $\{u-v<M\}$. Here we use
\begin{equation*}
\left\{ \begin{aligned}
 &\Delta u=u(u-\varphi)^2~~\text{in}~~\{u-v<M\}, \\
 &\Delta u^t=u^t(u^t-\varphi^t)^2~~\text{in}~~\{u-v<M\}, \\
 & u^t\geq u~~\text{on}~~\partial\{u-v<M\},
                          \end{aligned} \right.
\end{equation*}
where $\varphi=u-v$ and $\varphi^t=u^t-v^t$. Assume that
\[\inf_{\{u-v<M\}}\left(u^t-u\right)<0.\]
 First let us assume this minimum is attained at an interior point
$x_0\in\{u-v<M\}$. The maximum principle implies
\[0\leq\Delta(u^t-u)(x_0)=u^t(x_0)(u^t(x_0)-\varphi^t(x_0))^2-u(x_0)(u(x_0)-\varphi(x_0))^2<0,\]
which is a contradiction. Here we have used the following facts:
$0<u^t(x_0)<u(x_0)$, $u(x_0)>\varphi(x_0)$,
$u^t(x_0)>\varphi^t(x_0)$ and by \eqref{8.2}, $\varphi(x_0)<\varphi^t(x_0)$.

Next, if the minimum is not attained, we can use the method in Step
1 to reduce this case to the above case and we get a contradiction again. In conclusion we must have
\[u^t\geq u~~\text{in}~~\{u-v<M\}.\]
Hence for $t\in(t_0-\varepsilon,t_0)$, we still have \eqref{9.1}.
This is a contradiction with the definition of $t_0$ if $t_0>0$. In
other words, $t_0=0$.
\end{proof}

\section{Enlargement of the cone of monotonicity}
\numberwithin{equation}{section}
 \setcounter{equation}{0}

In this section $(u,v)$ still denotes a solution of \eqref{equation} on
$\mathbb{R}^n$, satisfying the linear growth condition \eqref{linear growth} and the minimizing condition \eqref{local energy minimizer}.
In the previous section we have proved that for every $\tau\in
C(e_n,3/4)$,
\[\frac{\partial u}{\partial\tau}\geq 0,\ \ \ \ \frac{\partial v}{\partial\tau}\leq 0~~\text{in}~~\mathbb{R}^n.\]
Now we enlarge this cone of monotonicity.
\begin{prop}
For every unit vector $\tau$ such that $\tau\cdot e_n=0$,
\[\frac{\partial u}{\partial\tau}\equiv0,\ \ \ \ \frac{\partial v}{\partial\tau}\equiv 0~~\text{in}~~\mathbb{R}^n.\]
\end{prop}

This means that $u$ and $v$ depend only on the $x_n$ variable and
finishes the proof of Theorem \ref{main result}.
\begin{proof}
For $\theta\in[0,\pi/2]$, denote
\[\tau(\theta)=\cos(\theta) e_n+\sin(\theta)\tau.\]
Define
\[I:=\{\theta\in[0,\frac{\pi}{2}]: \frac{\partial u}{\partial\tau(\theta)}\geq
0,\ \ \frac{\partial v}{\partial\tau(\theta)}\leq
0~~\text{in}~~\mathbb{R}^n\}.\]
 By Proposition \ref{lem 9.5},
$[0,\pi/100]\subset I$. We want to prove $I=[0,\pi/2]$. First we
have\\
{\bf Claim.} For every $\theta_0\in I$ and $\theta_0<\pi/2$, there
exists a constant $\delta>0$ such that
\[\frac{\partial (u-v)}{\partial\tau(\theta_0)}\geq \delta~~\text{in}~~\mathbb{R}^n.\]
By Lemma \ref{lem 9.1} (if we choose $M$ large enough, depending on
$\tau(\theta_0)\cdot e_n$), in $\{|u-v|\geq M\}$,
\[|\nabla(u-v)-e_n|\leq \frac{\tau(\theta_0)\cdot e_n}{2}.\]
Hence in $\{|u-v|\geq M\}$,
\[\frac{\partial (u-v)}{\partial\tau(\theta_0)}\geq \frac{\tau(\theta_0)\cdot e_n}{2}>0.\]
So the problem lies in the set $\{|u-v|\leq M\}$. Assume
\[\inf_{\{|u-v|\leq M\}}\frac{\partial (u-v)}{\partial\tau(\theta_0)}=0.\]
Take a minimizing sequence $x_k$ and proceed as in the proof of
Proposition \ref{lem 9.5}, we get a solution of \eqref{equation},
$(u_\infty, v_\infty)$ such that
\begin{enumerate}
\item $\frac{\partial u_\infty}{\partial\tau(\theta_0)}\geq 0$,
$\frac{\partial v_\infty}{\partial\tau(\theta_0)}\leq 0$ in
$\mathbb{R}^n$;
\item $\frac{\partial (u_\infty-v_\infty)}{\partial\tau(\theta_0)}\geq \frac{\tau(\theta)\cdot e_n}{2}>0$
in $\{|u-v|\geq M\}$;
\item $\frac{\partial
(u_\infty-v_\infty)}{\partial\tau(\theta_0)}(0)=0$, i.e.
$\frac{\partial u_\infty}{\partial\tau(\theta_0)}(0)=\frac{\partial
v_\infty}{\partial\tau(\theta_0)}(0)=0$.
\end{enumerate}
By differentiating the equations of $u_\infty$ and $v_\infty$, we
find
\begin{equation*}
\left\{ \begin{aligned}
 &\Delta \frac{\partial u_\infty}{\partial\tau(\theta_0)}
 =\frac{\partial u_\infty}{\partial\tau(\theta_0)}v_\infty^2
 +2u_\infty v_\infty\frac{\partial v_\infty}{\partial\tau(\theta_0)}, \\
 &\Delta \frac{\partial v_\infty}{\partial\tau(\theta_0)}
 =\frac{\partial v_\infty}{\partial\tau(\theta_0)}u_\infty^2
 +2u_\infty v_\infty\frac{\partial u_\infty}{\partial\tau(\theta_0)}.
                          \end{aligned} \right.
\end{equation*}
By the sign of $\frac{\partial u_\infty}{\partial\tau(\theta_0)}$
and $\frac{\partial v_\infty}{\partial\tau(\theta_0)}$, we can apply
the strong maximum principle to deduce that either
\[\frac{\partial u_\infty}{\partial\tau(\theta_0)}\equiv 0,\ \ \ \
\frac{\partial v_\infty}{\partial\tau(\theta_0)}\equiv
0~~\text{in}~~\mathbb{R}^n,\] which contradicts (2) in the above, or
\[\frac{\partial u_\infty}{\partial\tau(\theta_0)}> 0,\ \ \ \
\frac{\partial v_\infty}{\partial\tau(\theta_0)}<
0~~\text{in}~~\mathbb{R}^n,\] which contradicts (3) in the above.
This proves the claim.
\par
By the boundedness of $\nabla(u-v)$ and the above claim, for
$\theta_0\in I$ and $\theta_0<\pi/2$, there exists a $\varepsilon>0$
such that for every $\theta\in(\theta_0-\varepsilon,\theta_0]$,
\[\frac{\partial (u-v)}{\partial\tau(\theta)}\geq \frac{\delta}{2}~~\text{in}~~\mathbb{R}^n.\]
Similar to Step 2 in the proof of Proposition \ref{lem 9.5}, for
such $\theta$ we still have
\[ \frac{\partial u}{\partial\tau(\theta)}\geq 0,\ \ \ \  \frac{\partial
v}{\partial\tau(\theta)}\leq 0~~\text{in}~~\mathbb{R}^n.\] That is,
$(\theta_0-\varepsilon,\theta_0)\subset I$. By continuation, we get
$I=[0,\pi/2]$. In particular,
\[ \frac{\partial u}{\partial\tau}\geq 0,\ \ \ \  \frac{\partial
v}{\partial\tau}\leq 0~~\text{in}~~\mathbb{R}^n.\]
 Replacing $\tau$
by $-\tau$, we get
\[\frac{\partial u}{\partial\tau}\equiv0,\ \ \ \ \frac{\partial v}{\partial\tau}\equiv 0~~\text{in}~~\mathbb{R}^n.\qedhere\]
\end{proof}

\appendix
\section{Appendix}
\numberwithin{equation}{section}
 \setcounter{equation}{0}
Here we give a proof of Theorem \ref{lem uniform Holder}. Since the
method is exactly the one given in \cite{NTTV}, we only show the
necessary modifications.

Let $(u_\kappa,v_\kappa)$ be given as in Theorem \ref{lem uniform
Holder} and fix an $\alpha\in(0,1)$. Without loss of
generality, we assume that $u_\kappa$ and $v_\kappa$ are defined in
$B_3(0)$, and uniformly bounded there.

Take $\eta\in C^\infty(\mathbb{R}^n)$ such that $\eta\equiv 1$ in
$B_1(0)$, $\{\eta>0\}=B_2(0)$, and $\eta=(2-|x|)^2$ in
$B_2(0)\setminus B_{3/2}(0)$. By this choice we get a constant $C$
such that
\[|\nabla\log\eta|\leq C\text{dist}(x,\partial B_2(0))^{-1}~~\text{in}~~B_2(0).\]
Then for any $x\in B_2(0)$, by taking
$\rho:=\frac{1}{2}\text{dist}(x,\partial B_2(0))$, we have
\begin{equation}\label{A01}
\sup_{B_\rho(x)}\eta\leq C\inf_{B_\rho(x)}\eta,
\end{equation}
where $C$ is independent of $\rho$.

 Denote
\[\hat{u}_\kappa=u_\kappa\eta,\ \ \ \ \hat{v}_\kappa=v_\kappa\eta.\]
Assume there exist $x_\kappa,y_\kappa\in B_2(0)$ such that as
$\kappa\to+\infty$,
\begin{equation}\label{A1}
L_\kappa=\frac
{|\hat{u}_\kappa(x_\kappa)-\hat{u}_\kappa(y_\kappa)|+|\hat{v}_\kappa(x_\kappa)-\hat{v}_\kappa(y_\kappa)|}
{|x_\kappa-y_\kappa|^\alpha}=\max_{x,y\in B_2(0)}\frac
{|\hat{u}_\kappa(x)-\hat{u}_\kappa(y)|+|\hat{v}_\kappa(x)-\hat{v}_\kappa(y)|}
{|x-y|^\alpha}\to+\infty.
\end{equation}
Note that because $u_\kappa$ and $v_\kappa$ are uniformly bounded,
as $\kappa\to+\infty$, $|x_\kappa-y_\kappa|\to 0$.

Define
\[\widetilde{u}_\kappa(x):=L_\kappa^{-1}r_\kappa^{-\alpha}\hat{u}_\kappa(x_\kappa+r_\kappa
x),\ \ \ \ \widetilde{v}_\kappa(x):=L_\kappa^{-1}r_\kappa^{-\alpha}\hat{v}_\kappa(x_\kappa+r_\kappa x),
\]
and
\[\bar{u}_\kappa(x):=L_\kappa^{-1}r_\kappa^{-\alpha}u_\kappa(x_\kappa+r_\kappa
x)\eta(x_\kappa),\ \ \ \ \bar{v}_\kappa(x):=L_\kappa^{-1}r_\kappa^{-\alpha}v_\kappa(x_\kappa+r_\kappa x)\eta(x_\kappa),
\]
where $r_\kappa\to 0$ will be determined later. These functions are
defined in $\Omega_\kappa=r_\kappa^{-1}(B_2(0)-x_\kappa)$. Note
that $\Omega_\kappa$ converges to $\Omega_\infty$, which may be the
entire space or an half space.

We first present some facts about these rescaled functions, which
will be used below. By definition we have
\begin{equation}\label{A02}
\widetilde{u}_\kappa(x)=\frac{\eta(x_\kappa+r_\kappa
x)}{\eta(x_\kappa)}\bar{u}_\kappa(x),
\end{equation}
and
\begin{eqnarray}\label{A03}
\nabla\widetilde{u}_\kappa(x)&=&\frac{r_\kappa\nabla\eta(x_\kappa+r_\kappa
x)}{\eta(x_\kappa)}\bar{u}_\kappa(x)+\frac{\eta(x_\kappa+r_\kappa
x)}{\eta(x_\kappa)}\nabla\bar{u}_\kappa(x)\\\nonumber
&=&L_\kappa^{-1}r_\kappa^{1-\alpha}u_\kappa(x_\kappa+r_\kappa
x)\nabla\eta(x_\kappa+r_\kappa x) +\frac{\eta(x_\kappa+r_\kappa
x)}{\eta(x_\kappa)}\nabla\bar{u}_\kappa(x)\\\nonumber
&=&\frac{\eta(x_\kappa+r_\kappa
x)}{\eta(x_\kappa)}\nabla\bar{u}_\kappa(x)+O(L_\kappa^{-1}r_\kappa^{1-\alpha}).
\end{eqnarray}
These relations also hold for $\widetilde{v}_\kappa$ and
$\bar{v}_\kappa$.

By \eqref{A1}, we have
\begin{equation}\label{A2}
1=\frac
{|\widetilde{u}_\kappa(0)-\widetilde{u}_\kappa(z_\kappa)|+|\widetilde{v}_\kappa(0)-\widetilde{v}_\kappa(z_\kappa)|}
{|z_\kappa|^\alpha}=\max_{x,y\in \Omega_\kappa}\frac
{|\widetilde{u}_\kappa(x)-\widetilde{u}_\kappa(y)|+|\widetilde{v}_\kappa(x)-\widetilde{v}_\kappa(y)|}
{|x-y|^\alpha}.
\end{equation}
Here $z_\kappa=\frac{y_\kappa-x_\kappa}{r_\kappa}$.

 Next, because
$\eta$ is Lipschitz continuous in $\overline{B_2(0)}$, for
$x\in\Omega_\kappa$, we have a constant $C$ which depends only on
$\sup_{B_2(0)}(u_\kappa+v_\kappa)$ and the Lipschitz constant of
$\eta$, such that
\begin{eqnarray}\label{A3}
|\widetilde{u}_\kappa(x)-\bar{u}_\kappa(x)|+|\widetilde{v}_\kappa(x)-\bar{v}_\kappa(x)|&\leq&CL_\kappa^{-1}r_\kappa^{-\alpha}|\eta(x_\kappa+r_\kappa
x)-\eta(x_\kappa)|\\\nonumber &\leq&
CL_\kappa^{-1}r_\kappa^{1-\alpha}|x|.
\end{eqnarray}
This converges to $0$ uniformly on any compact set as
$\kappa\to+\infty$. By the Lipschitz continuity of $\eta$, we also
have
\begin{equation}\label{A4}
\widetilde{u}_\kappa(x)+\widetilde{v}_\kappa(x)\leq
CL_\kappa^{-1}r_\kappa^{1-\alpha}\text{dist}(x,\partial\Omega_\kappa).
\end{equation}
Finally, we note that
$(\bar{u}_\kappa,\bar{v}_\kappa)$ satisfies \eqref{equation scaled}
with the parameter $\kappa
L_\kappa^2r_\kappa^{2+\alpha}\eta(x_\kappa)^{-2}$. Denote this
constant by $M_\kappa$.

The next three lemmas correspond to Lemma 3.4, 3.5 and 3.6 in
\cite{NTTV} respectively.
\begin{lem}
If there exists a constant $C>0$ such that $M_\kappa\geq
\frac{1}{C}$ and $\frac{|x_\kappa-y_\kappa|}{r_\kappa}\leq C$, then
$\widetilde{u}_\kappa(0)+\widetilde{v}_\kappa(0)$ is uniformly
bounded.
\end{lem}
\begin{proof}
Assume under the assumptions in this lemma, we have
\[A_\kappa:=\widetilde{u}_\kappa(0)=\bar{u}_\kappa(0)\to+\infty.\]
By \eqref{A4},
\begin{equation}\label{A5}
\text{dist}(0,\partial\Omega_\kappa)\geq CL_\kappa
r_\kappa^{-1+\alpha}A_\kappa,
\end{equation}
 which goes to
$+\infty$ as $\kappa\to+\infty$. Hence, for any $R>0$, if $\kappa$
is sufficiently large, by \eqref{A2} we have
\[\widetilde{u}_\kappa(x)\geq A_\kappa-2R^\alpha,~~\text{in}~~B_{2R}(0).\]
By \eqref{A3}, for fixed $R$, if $\kappa$ is sufficiently large, we
have
\[\bar{u}_\kappa(x)\geq A_\kappa-2R^\alpha-CL_\kappa^{-1}r_\kappa^{1-\alpha}R\geq\frac{A_\kappa}{2},
~~\text{in}~~B_{2R}(0).\]
Hence
\[\Delta \bar{v}_\kappa\geq CM_\kappa A_\kappa^2\bar{v}_\kappa~~\text{in}~~B_{2R}(0).\]
 By \eqref{A2} and Lemma \ref{estimate exponential decay}, we get
\[\sup_{B_R(0)}\bar{v}_\kappa\leq C\left(\sup_{B_R(0)}\bar{v}_\kappa+R^\alpha
+CL_\kappa^{-1}r_\kappa^{1-\alpha}R\right)e^{-\frac{1}{C}RM_\kappa^{1/2}A_\kappa}.\]
This implies
\[\sup_{B_R(0)}\bar{v}_\kappa\leq CR^\alpha e^{-\frac{1}{C}RM_\kappa^{1/2}A_\kappa}.\]
Substituting this into the equations of $\bar{u}_\kappa$ and
$\bar{v}_\kappa$, we see
\[\sup_{B_R(0)}\left(\Delta\bar{u}_\kappa+\Delta\bar{v}_\kappa\right)\to 0.\]
By standard elliptic estimates combined with \eqref{A2},
\eqref{A3} and the equation of $\bar{u}_\kappa$, $\bar{v}_\kappa$, we obtain
\[\sup_{B_{R-1}(0)}\left(|\nabla\bar{u}_\kappa|+|\nabla\bar{v}_\kappa|\right)\leq C(R),\]
where $C(R)$ depends only on $R$. Note that \eqref{A5} implies
\[\text{dist}(x_\kappa,\partial B_2(0))\gg r_\kappa.\]
In particular, $y_\kappa\in B_{\text{dist}(x_\kappa,\partial
B_2(0))/2}(x_\kappa)$. Then by combining \eqref{A01} with
\eqref{A03}, we get for another constant $C(R)$,
\[\sup_{B_{R-1}(0)}\left(|\nabla\widetilde{u}_\kappa|+|\nabla\widetilde{v}_\kappa|\right)\leq C(R).\]
By our assumption, we can take $R$ large so that $|z_\kappa|\leq R$
for all $\kappa>0$. Substituting the above estimate into \eqref{A2},
we see
\[1=\frac
{|\widetilde{u}_\kappa(0)-\widetilde{u}_\kappa(z_\kappa)|+|\widetilde{v}_\kappa(0)-\widetilde{v}_\kappa(z_\kappa)|}
{|z_\kappa|^\alpha}\leq C(R)|z_\kappa|^{1-\alpha},\] which implies a
uniform lower bound of $|z_\kappa|$. Then arguing as in the proof of
Lemma 3.4 in \cite{NTTV}, we can get a contradiction and finish the
proof.
\end{proof}

\begin{lem}
$\kappa
L_\kappa^2|x_\kappa-y_\kappa|^{2+\alpha}\eta(x_\kappa)^{-2}\to+\infty$.
\end{lem}
\begin{proof}
By contradiction, assume that $\kappa
L_\kappa^2|x_\kappa-y_\kappa|^{2+\alpha}\eta(x_\kappa)^{-2}$ is
bounded. Take $r_\kappa$ so that
\[\kappa
L_\kappa^2r_\kappa^{2+\alpha}\eta(x_\kappa)^{-2}=1.\] By this
choice, the assumptions of the previous lemma are satisfied. Hence
$\{\widetilde{u}_\kappa(0)\},\{\widetilde{v}_\kappa(0)\}$ are
bounded. By \eqref{A2} and the Ascoli-Arzel\`{a} theorem, up to a
subsequence, there exist $u_\infty$, $v_\infty$ such that
$\widetilde{u}_\kappa\to u_\infty$, $\widetilde{v}_\kappa\to
v_\infty$ uniformly on any compact set of $\Omega_\infty$. By
\eqref{A3}, $\bar{u}_\kappa\to u_\infty$, $\bar{v}_\kappa\to
v_\infty$ uniformly on any compact set of $\Omega_\infty$, too.
Moreover, since $(\bar{u}_\kappa,\bar{v}_\kappa)$ is defined in
$B_{r_\kappa^{-1}}(0)$ and satisfies \eqref{equation scaled} with
parameter $1$, for any $R>0$,
\begin{equation}\label{A6}
\sup_{B_{R}(0)}\left(|\nabla\bar{u}_\kappa|+|\nabla\bar{v}_\kappa|\right)\leq
C(R).
\end{equation}
 If
$\text{dist}(0,\partial\Omega_\kappa)\to+\infty$ (i.e.
$\Omega_\infty$ is the entire space), we can argue as in the
previous lemma to deduce that
\[\sup_{B_{R}(0)}\left(|\nabla\widetilde{u}_\kappa|+|\nabla\widetilde{v}_\kappa|\right)\leq C(R).\]
which, as before, implies a uniform lower bound of $|z_\kappa|$.
Then the following proof is exactly the proof of Lemma 3.5 in
\cite{NTTV}.

Next we consider the case when $\Omega_\infty$ is an half-space, that
is, there exists a constant $C>0$ such that
\[\text{dist}(0,\partial\Omega_\kappa)\leq C.\]
We still need a uniform lower bound of $|z_\kappa|$ as above. Assume
by contrary that $|z_\kappa|\to 0$. Then by \eqref{A1} and
\eqref{A4}, we have
\begin{eqnarray*}
|z_\kappa|^\alpha&=&
|\widetilde{u}_\kappa(0)-\widetilde{u}_\kappa(z_\kappa)|
+|\widetilde{v}_\kappa(0)-\widetilde{v}_\kappa(z_\kappa)|\\
&\leq&
CL_\kappa^{-1}r_\kappa^{1-\alpha}\left(\text{dist}(0,\partial\Omega_\kappa)
+\text{dist}(z_\kappa,\partial\Omega_\kappa)\right)\\
&\leq&
CL_\kappa^{-1}r_\kappa^{1-\alpha}\left(\text{dist}(0,\partial\Omega_\kappa)
+|z_\kappa|\right).
\end{eqnarray*}
Since $|z_\kappa|\to 0$ and $\alpha\in(0,1)$, this implies
\[|z_\kappa|\leq CL_\kappa^{-1}r_\kappa^{1-\alpha}\text{dist}(0,\partial\Omega_\kappa)
<\frac{1}{2}\text{dist}(0,\partial\Omega_\kappa).\] Arguing as
before and using \eqref{A6} and \eqref{A03}, we can get
\[\sup_{B_{\frac{\text{dist}(0,\partial\Omega_\kappa)}{2}}(0)}
\left(|\nabla\widetilde{u}_\kappa|+|\nabla\widetilde{v}_\kappa|\right)\leq C.\]
Then exactly as in the previous lemma, we get a constant $c>0$ such
that $|z_\kappa|\geq c$ for all $\kappa$. So in any case, we must
have a uniform lower bound for $|z_\kappa|$. Combining this with
\eqref{A2} and \eqref{A4}, we get a contradiction directly in this
case.
\end{proof}

\begin{lem}
Let $r_\kappa=|x_\kappa-y_\kappa|$. Then there exist $u_\infty$,
$v_\infty\in C^\alpha(\mathbb{R}^n)$ such that $\bar{u}_\kappa\to
u_\infty$, $\bar{v}_\kappa\to v_\infty$ uniformly on any compact set
of $\mathbb{R}^n$. The limit satisfies $u_\infty v_\infty\equiv 0$.
\end{lem}
\begin{proof}
By the previous lemma, we must have \[M_\kappa=\kappa
L_\kappa^2|x_\kappa-y_\kappa|^{2+\alpha}\eta(x_\kappa)^{-2}\to+\infty.\]
Hence if we choose $r_\kappa=|x_\kappa-y_\kappa|$,
$\{\widetilde{u}_\kappa(0)\},\{\widetilde{v}_\kappa(0)\}$ are
bounded. As before, when $\Omega_\infty$ is a half-space, combining
\eqref{A2} and \eqref{A4} we can get a contradiction directly. So
$\Omega_\infty$ must be the entire space $\mathbb{R}^n$. Then as in
the previous lemma, there exist $u_\infty$, $v_\infty\in
C^\alpha(\mathbb{R}^n)$ such that $\bar{u}_\kappa\to u_\infty$,
$\bar{v}_\kappa\to v_\infty$ uniformly on any compact set of
$\mathbb{R}^n$. We can follow the proof of Lemma 3.6 in
\cite{NTTV} to get the claim.
\end{proof}

Using the Liouville type results in \cite{NTTV} together with the
above three lemmas, we know \eqref{A1} can not be true. So
$\hat{u}_\kappa$ and $\hat{v}_\kappa$ are uniformly bounded in
$C^\alpha(B_2(0))$. Since $\eta\equiv 1$ in $B_1(0)$, it is easily
seen that $u_\kappa$ and $v_\kappa$ are uniformly bounded in
$C^\alpha(B_1(0))$, for any $\alpha\in (0,1)$. This finishes the
proof of Theorem \ref{lem uniform Holder}.

\begin{rmk}
We can allow some additional terms in the right hand side of
equations in \eqref{equation scaled} as in \cite{NTTV}. Theorem
\ref{lem uniform Holder} can also be generalized to the case with
more than two equations, and to the case of parabolic equations
(corresponding to the main result in \cite{D-W-Z 2}). This method
can also be applied to obtain a local estimate near the boundary if the
boundary and boundary values are sufficiently smooth.
\end{rmk}

\end{document}